\documentclass[12pt]{amsart}
\usepackage{amssymb,euscript,amsmath, mathrsfs, amscd}
\usepackage[dvips]{graphicx, rotating}
\usepackage[dvips]{color}
\usepackage{epsfig}

\newcounter{ENUM}
\newcommand{\itm}{\item}
\newenvironment{ilist}{\renewcommand{\theENUM}{\roman{ENUM}}\renewcommand{\itm}{\addtocounter{ENUM}{1}\item[(\theENUM)]}\begin{itemize}\setcounter{ENUM}{0}}{\end{itemize}}
\newenvironment{alist}[1][0]{\renewcommand{\theENUM}{\alph{ENUM}}\renewcommand{\itm}{\addtocounter{ENUM}{1}\item[\theENUM)]}\begin{itemize}\setcounter{ENUM}{#1}}{\end{itemize}}

\newcommand{\nBar}[1]{\overline{#1}}

\def\ds{\displaystyle}
\newcommand{\bm}[1]{{\boldsymbol{#1}}}
\def\0{\bm{0}}
\def\1{\bm{1}}

\def\bc{\bm{c}}

\def\be{\bm{e}}

\def\bh{\bm{h}}

\def\bn{\bm{n}}

\def\t{\bm{t}}
\def\v{\bm{v}}
\def\u{\bm{u}}
\def\w{\bm{w}}
\def\x{\bm{x}}
\def\y{\bm{y}}
\def\z{\bm{z}}

\def\bbeta{\bm{\beta}}

\def\C{{\mathbb C}}
\def\CP{{\mathbb {CP}}}

\def\N{{\mathbb N}}
\def\P{{\mathbb P}}

\def\R{{\mathbb R}}
\def\Z{{\mathbb Z}}

\def\sP{{\mathscr P}}

\def\cN{\mathcal N}
\def\O{{\mathcal O}}
\def\cP{{\mathcal P}}
\def\cQ{{\mathcal Q}}

\def\cS{{\mathcal S}}

\def\supp{\operatorname{supp}}

\def\maxv{\mathrm{maxv}}
\def\minv{\mathrm{minv}}
\def\irr{\mathrm{irr}}

\def\maxv{\mathrm{maxv}}
\def\minv{\mathrm{minv}}

\newtheorem{thm}{Theorem}[section]
\newtheorem{prop}[thm]{Proposition}
\newtheorem{lem}[thm]{Lemma}
\newtheorem{cor}[thm]{Corollary}

\theoremstyle{definition}
\newtheorem{defn}[thm]{Definition}

\newtheorem{ex}[thm]{Example}

\theoremstyle{remark}

\newtheorem{rem}[thm]{Remark}

\numberwithin{equation}{section}

\usepackage{fullpage}
\allowdisplaybreaks

\def\ext{\operatorname{ext}}

\def\0{\bm 0}
\def\l{\ell}
\def\olambda{\overline{\lambda}}

\newcommand{\oo}{\multiput(0,0)(10,0){2}{\circle*{2}}}
\newcommand{\ooo}{\multiput(0,0)(10,0){3}{\circle*{2}}}
\newcommand{\oooo}{\multiput(0,0)(10,0){4}{\circle*{2}}}

\newcommand{\Eeee}{\put(1,0){\line(1,0){8}}}
\newcommand{\eEee}{\put(11,0){\line(1,0){8}}}

\begin{document}
\title{A combinatorial analysis of Severi degrees}
\author{Fu Liu}\thanks{Fu Liu is partially supported by a grant from the Simons Foundation \#245939 and by NSF grant DMS-1265702.} 
\address{Fu Liu, Department of Mathematics, University of California, Davis, One Shields Avenue, Davis, CA 95616 USA.}
\email{fuliu@math.ucdavis.edu}
\date{\today}

\begin{abstract}
  Based on results by Brugall\'e and Mikhalkin, Fomin and Mikhalkin give formulas for computing classical Severi degrees $N^{d, \delta}$ using long-edge graphs. In 2012, Block, Colley and Kennedy considered the logarithmic version of a special function associated to long-edge graphs appeared in Fomin-Mikhalkin's formula, and conjectured it to be linear. They have since proved their conjecture. At the same time, motivated by their conjecture, we consider a special multivariate function associated to long-edge graphs that generalizes their function. The main result of this paper is that the multivariate function we define is always linear. A special case of our result gives an independent proof of Block-Colley-Kennedy's conjecture.

  The first application of our linearity result is that by applying it to classical Severi degrees, we recover quadraticity of $Q^{d, \delta}$ and a bound $\delta$ for the threshold of polynomiality of $N^{d, \delta}.$ 
  Next, in joint work with Osserman, we apply the linearity result to a special family of toric surfaces and obtain universal polynomial results having connections to the G\"ottsche-Yau-Zaslow formula. As a result, we provide combinatorial formulas for the two unidentified power series $B_1(q)$ and $B_2(q)$ appearing in the G\"ottsche-Yau-Zaslow formula. 

  The proof of our linearity result is completely combinatorial. 
  We define $\tau$-graphs which generalize long-edge graphs, and a closely related family of combinatorial objects we call $(\tau, \bn)$-words. By introducing height functions and a concept of irreducibility, we describe ways to decompose certain families of $(\tau, \bn)$-words into irreducible words, which leads to the desired results. 
\end{abstract}

\maketitle

\section{Introduction}

\subsection{Background on Severi degrees}

The classical {\it Severi degree}, denoted by $N^{d,\delta},$ is the degree of the Severi variety. It counts the number of curves of degree $d$ with $\delta$ nodes passing through $\displaystyle \frac{d(d+3)}{2}-\delta$ general points in the complex porjective plane $\CP^2$. If $d \ge \delta +2,$ the Severi degree $N^{d, \delta}$ coincides with the Gromov-Witten invariant $N_{d, \frac{(d-1)(d-2)}{2} - \delta},$ which counts maps from curves to the plane. The problem of studying the Severi degrees dates back to late 19th century by Chasles, Zeuthen and Schubert. The modern study of the Severi variety was initiated by Harris' proof of their irreducibility \cite{harris1986}.

In 1994, Di Francesco and Itzykson \cite{DiFraItz1995} conjectured that for fixed $\delta,$ the Severi degree $N^{d, \delta}$ is given by a {\it node polynomial} $N_\delta(d)$ for sufficiently large $d$. In 2009, Fomin and Mikhalkin \cite[Theorem 5.1]{FomMik2010} established the polynomiality of $N^{d, \delta}$ using tropical geometry and floor decomposition. Since then Block has computed the node polynomial $N_\delta(d)$ up to $\delta = 14$ \cite{block2011}. The {\it threshold} of the polynomiality of $N^{d, \delta}$ is the value $d^*=d^*(\delta)$ such that $N^{d, \delta} = N_\delta(d)$ for all $d \ge d^*.$ 
Fomin and Mikhalkin \cite{FomMik2010} showed that $d^* \le 2 \delta;$ Block \cite{block2011} lowered it to $d^* \le \delta;$ and most recently Kleiman and Shende \cite{KleShe2012} proved the bound $d^* \le \lceil \delta/2 \rceil +1$ conjectured by G\"ottsche. 


Instead of restricting the attention to $\CP^2,$ one can ask same question of enumerating curves on other surfaces. Let $L$ be a line bundle on a complex projective smooth surface $S.$ 
We denote by $N^{\delta}(S, L)$ the number of $\delta$-nodal curves in $|L|$ passing through $\dim |L| - \delta$ points in general position. When $S = \CP^2$ and $L = \O_{\CP^2}(d),$ we recover the classical Severi degree $N^{d, \delta}.$ Hence, we can consider $N^{\delta}(S,L)$ to be a {\it generalized Severi degree}. 
In \cite[Conjecture 2.1]{gottsche1998}, G\"ottsche conjectured that for every $\delta$, there exists a universal polynomial $T_\delta(x,y,z,w)$ of degree $\delta$ that computes the numbers $N^\delta(S, L)$ 
by evaluating $T_\delta$ at the four topological numbers of $(S, L)$: $L^2, LK_S, K_S^2$ and $c_2(S),$ provided that the line bundle $L$ is $(5\delta-1)$-very ample. Furthermore, inspired by the Yau-Zaslow formula, G\"ottsche \cite[Conjecture 2.4]{gottsche1998} conjectured the closed form of the generating function of $T_\delta$, which is known as the G\"ottsche-Yau-Zaslow formula. 
Recently, Tzeng \cite{tzeng2012} and Kool-Shende-Thomas \cite{KooSheTho2011} independently proved G\"ottsche's conjectures. Note that in the case of $\CP^2,$ the four topological numbers become: $L^2 = d^2, LK_S = -3d, K_S^2 = 9$ and $c_2(S)=3.$ Thus,  
\begin{equation}\label{equ:T2N}
  T_\delta(d^2, -3d, 9, 3) = N_\delta(d).
\end{equation}

In \cite{gottsche1998}, G\"ottsche discussed a consequence of the G\"ottsche-Yau-Zaslow's formula.
\begin{prop}[\cite{gottsche1998}, Proposition 2.3] \label{prop:logT}
  There exist four universal power series $A_1(t), A_2(t), A_3(t)$ and $A_4(t)$ such that 
\[ \log \left(\sum_{\delta \ge 0} T_\delta(x,y,z,w) t^\delta \right) = x A_1(t) + y A_2(t) + z A_3(t) + w A_4(t).\]
\end{prop}
This means that the coefficient of $t^\delta$ in the formal logarithm of $\sum_{\delta \ge 0} T_\delta(x,y,z,w) t^\delta$ is a linear function in $x, y, z$ and $w,$ which is potentially simpler than the expression for $T_\delta(x,y,z,w).$

Therefore, it is natural for us to consider the generating function for classical Severi degrees:
\begin{equation}\label{equ:defncN} \cN(d) := 1 + \sum_{\delta \ge 1} N^{d, \delta} t^{\delta},\end{equation}
and its formal logarithm 
\begin{equation}\label{equ:defncQ} \cQ(d) := \log(\cN(d))  = \sum_{\delta \ge 1} Q^{d, \delta} t^{\delta}. \end{equation}
%
%
%
It is straightforward to show that $Q^{d, \delta}$ is also a polynomial in $d$ for sufficiently large $d.$ We denote this polynomial by $Q_\delta(d).$ It is clear that
\[ \log\left(\sum_{\delta \ge 0} N_\delta(d) t^\delta \right) = \sum_{\delta \ge 1} Q_\delta(d) t^\delta .\]

Although the degree of $N_\delta(d)$ was shown to be $2 \delta$, the polynomial $Q_\delta(d)$, which is an alternating sum of $N_\delta(d)$'s, turns out to be quadratic, following from \eqref{equ:T2N} and Proposition \ref{prop:logT}. (See Proposition 3.1 in \cite{qviller2012}.)  
\begin{cor}
  \label{cor:quadratic}
	For any fixed $\delta,$ $Q^{d, \delta}$ is a quadratic polynomial in $d$ for sufficiently large $d.$
\end{cor}
In this paper, we will provide another proof of Corollary \ref{cor:quadratic} as well as a combinatorial way of computing the power series $A_1(t)$ and $A_2(t)$ by proving a certain function associated to {\it long-edge graphs} is linear. 
We give a brief introduction to the objects in our results below, and will fill in the details in Section \ref{sec:vialongedge}.

\subsection{Long-edge graphs and the main result}
Brugall\'e and Mikhalkin \cite{BruMik2007, BruMik2009} introduced ``(marked) labeled floor diagrams'' and gave an enumerative formula for the Severi degree $N^{d, \delta}$ in terms of these diagrams. Fomin and Mikhalkin \cite{FomMik2010} reformulated Brugall\'e and Mikhalkin's results by introducing a ``template decomposition'' of labeled floor diagrams. They first constructed a bijection between labeled floor diagrams and {\it long-edge graphs} and then gave a natural decomposition of long-edge graphs into ``templates'' (Fomin and Mikhalkin did not name the graphs they use; the terminology ``long-edge graphs'' was first introduced in \cite{BloColKen2013}.)
\begin{defn}\label{defn:long-edge}
  A {\it long-edge} graph $G$ is a graph $(V,E)$ with a weight function $\rho$ satisfying the following conditions:
  \begin{alist}
	\itm The vertex set $V = \N =\{0, 1, 2, \dots\},$ and the edge set $E$ is finite.
	\itm Multiple edges are allowed, but loops are not.
	\itm The weight function $\rho: E \to \P$ assigns a positive integer to each edge.
	\itm There are no {\it short edge}, i.e., there's no edge connecting $i$ and $i+1$ with weight $1.$
\end{alist}
\end{defn}

We often draw the vertices $0, 1, 2, \dots$ of long-edge graphs from left to right and label each edge with its weight. 
Since all but finitely many vertices do not have incident edges, we often omit most of irrelevant vertices when we draw long-edge graphs.
See Figure \ref{fig:exlongedge} for three examples of long-edge graphs.

\begin{figure}
		\begin{picture}(440,60)(0,-15)
\setlength{\unitlength}{4.0pt}\thicklines
\multiput(0,0)(10,0){3}{\circle*{2}}
\put(1,0){\line(1,0){8}}
\qbezier(0.8,0.6)(10,10)(19.2,0.6)
\put(5,-2){\makebox(0,0){$\scriptstyle 2$}}
\put(10,7){\makebox(0,0){$\scriptstyle 1$}}
\put(0,-3){\makebox(0,0){$0$}}
\put(10,-3){\makebox(0,0){$1$}}
\put(20,-3){\makebox(0,0){$2$}}
\put(2,10){\makebox(0,0){$G_1$}}

\multiput(40,0)(10,0){3}{\circle*{2}}
\put(41,0){\line(1,0){8}}
\qbezier(40.8,0.6)(50,10)(59.2,0.6)
\put(45,-2){\makebox(0,0){$\scriptstyle 2$}}
\put(50,7){\makebox(0,0){$\scriptstyle 1$}}
\put(40,-3){\makebox(0,0){$3$}}
\put(50,-3){\makebox(0,0){$4$}}
\put(60,-3){\makebox(0,0){$5$}}
\put(42,10){\makebox(0,0){$G_2$}}

\multiput(80,0)(10,0){4}{\circle*{2}}
\put(81,0){\line(1,0){8}}
\qbezier(80.8,0.6)(90,10)(99.2,0.6)
\put(101,0){\line(1,0){8}}
\put(85,-2){\makebox(0,0){$\scriptstyle 2$}}
\put(90,7){\makebox(0,0){$\scriptstyle 1$}}
\put(105,-2){\makebox(0,0){$\scriptstyle 2$}}
\put(80,-3){\makebox(0,0){$3$}}
\put(90,-3){\makebox(0,0){$4$}}
\put(100,-3){\makebox(0,0){$5$}}
\put(110,-3){\makebox(0,0){$6$}}
\put(82,10){\makebox(0,0){$G_3$}}
\end{picture}
\caption{Examples of long-edge graphs}
\label{fig:exlongedge}
\end{figure}
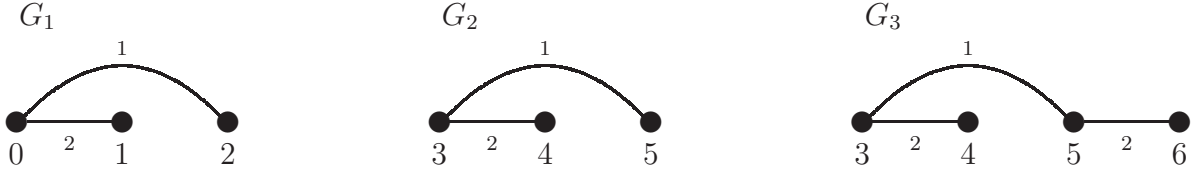

Fomin and Mikhalkin associate to each long-edge graph a statistic $\nu(G)$, and then give an enumerative formula for computing the Severi degree $N^{d, \delta}$ in terms of long-edge graphs using this statistic. We extend their definition and define two closely related statistics $P_\bbeta(G)$ and $P_\bbeta^s(G)$ for any given finite sequence $\bbeta,$ where $P_{(0,1,\dots, d)}^s(G)$ is the same as $\nu(G)$ defined by Fomin and Mikhalkin. We consider logarithmic versions of $P_\bbeta(G)$ and $P_\bbeta^s(G)$. For any long-edge graph $G$,  we define
\begin{equation}\label{equ:qG0}
\Phi_\bbeta(G) := \sum_{i \ge 1} \frac{(-1)^{i+1}}{i} \sum_{(G_1, \dots, G_i)}  \left(\prod_{j=1}^i P_\bbeta(G_j) \right),
\end{equation}
\begin{equation}\label{equ:qdG0}
\Phi_\bbeta^s(G) := \sum_{i \ge 1} \frac{(-1)^{i+1}}{i} \sum_{(G_1, \dots, G_i)}  \left(\prod_{j=1}^i P_\bbeta^s(G_j) \right),
\end{equation}
where both summations are over all the partitions of $G.$

Below is the main result of this paper.
\begin{thm}\label{thm:main0intro}
	Suppose $G$ is a long-edge graph. Then $\Phi_\bbeta\left( G \right)$ is a linear function in $\bbeta$ for sufficiently large $\bbeta.$
\end{thm}
 
The proof of Theorem \ref{thm:main0intro} is purely combinatorial and provides combinatorial objects to compute the coefficients of the linear function described in the theorem. We note that in 2012, the special case of Theorem \ref{thm:main0intro} when $\bbeta=(0,1,2,\dots, d)$ was conjectured by Block, Colley and Kennedy. They have since given in \cite{BloColKen2013} an independent proof of their conjecture. In fact, the original motivation of this paper was to prove their conjecture. However, the author noticed that the combinatorial approaches presented in this paper can be easily extended to prove our main result which has a lot more applications. 

 Below we discuss briefly applications of our main result.

\subsection{Applications of the main result}

The first application of our theorem is that, as in  \cite{BloColKen2013}, we can provide another proof of Corollary \ref{cor:quadratic} by applying it to the logarithmic version of Fomin-Mikhalkin's formula for classical Severi degrees. However, our techniques go further providing a new method for computing $Q^{d, \delta}$ and $N^{d,\delta}$. We are also able to recover the threshold bound $d^* \le \delta$ given by Block (see Remark \ref{rem:threshold}).  
Moreover, in the process of proving Corollary \ref{cor:quadratic}, we discover combinatorial formulas involving the coefficients of the linear function described in Theorem \ref{thm:main0intro} for computing the power series $A_1(t)$ and $A_2(t)$ of Proposition \ref{prop:logT}.

More importantly, the enumerative formula using labeled floor diagrams for Severi degrees introduced by Brugall\'e and Mikhalkin in \cite{BruMik2007, BruMik2009} does not only work for complex projective planes $\CP^2$, it also works for a more general family of (not necessarily smooth) toric surfaces coming from ``h-transverse'' polygons. In joint work with Osserman \cite{toricsurface}, we apply Theorem \ref{thm:main0intro} to this family of toric surfaces and obtain universal polynomial results having close connection to the G\"ottsche-Yau-Zaslow formula. 
Furthermore, results in \cite{toricsurface} provide a simpler combinatorial formula for computing $A_2(t)$ as well as a combinatorial formula for computing $A_3(t)-A_4(t)$ (where $A_i(t)$'s are the power series described in Proposition \ref{prop:logT}).
There are four power series involved in the G\"ottsche-Yau-Zaslow formula. Although two of the four power series, which often referred to as $B_1(q)$ and $B_2(q)$, are not explicitly identified, their terms can be computed by using the recursive formula of \cite{CapHar1998} for the classical Severi degrees $N^{d, \delta}$ and applying interpolation methods as soon as the threshold $d^*$ of the polynomiality of $N^{d,\delta}$ is known. Assuming the threshold bound $d^* \le \lceil \delta/2\rceil+1,$ G\"ottsche \cite[Remark 2.5]{gottsche1998} has calculated $B_1(q)$ and $B_2(q)$ up to degree $28$. 
Since $A_2(t)$ and $A_3(t)-A_4(t)$ determines $B_2(q)$ and $B_1(q)$ respectively, the paper \cite{toricsurface} provides combinatorial methods for computing $B_1(q)$ and $B_2(q)$ directly without using interpolation.

We won't discuss in this paper details of results and proofs in \cite{toricsurface}, which has a different focus, and are much more complicated than applications obtained by applying Theorem \ref{thm:main0intro} to the classical Severi degrees presented in this paper. However, our discussion on classical Severi degrees gives a demonstration of some ideas in \cite{toricsurface}, and also gives results that are not covered in \cite{toricsurface} such as a new method for computing $Q^{d, \delta}.$


This paper is organized as follows. 
\begin{ilist}
  \itm Section \ref{sec:vialongedge} is an {\it extended introduction}, in which we give detailed description of relevant functions for long-edge graphs and restate Theorem \ref{thm:main0intro} with more details (Theorem \ref{thm:main0}), followed by Fomin-Mikhalkin's formula for Severi degree and an analysis of its logarithmic form. 
	\itm Section \ref{sec:computeQ} and \ref{sec:determine} discusses the {\it applications} of our main result to the classical Severi degrees. In Section \ref{sec:computeQ}, we prove Corollary \ref{cor:quadratic} using Theorem \ref{thm:main0intro}, and give examples of how we can use the linear function described in Theorem \ref{thm:main0intro} to compute the quadratic polynomial $Q_\delta(d).$ 
In Section \ref{sec:determine}, we provide combinatorial formulas for $A_1(t)$ and $A_2(t)$ of Proposition \ref{prop:logT} by analyzing the formula for computing $Q_\delta(d)$ further.

\itm The rest of the paper mainly devotes to {\it the proof of Theorem \ref{thm:main0}}.

The first part of the proof consists of two reduction steps appearing in Sections \ref{sec:reformulate} and \ref{sec:polyreci}. In section \ref{sec:reformulate}, we introduce $\tau$-graphs, a generalization of long-edge graphs, and state our theorem in terms of $\tau$-graphs. We then reduce the problem to proving a theorem on the generating function on $\tau$-graphs (Theorem \ref{thm:main1}).
In Section \ref{sec:polyreci}, we introduce another combinatorial object: $(\tau, \bn)$-words, a special family of which, denoted by $S_\tau(\bn, \t)$, has a reciprocity connection to $\tau$-graphs. Using this connection, we reduce our problem (of proving Theorem \ref{thm:main1}) to proving a result on the generating function of $S_\tau(\bn, \t)$ (Theorem \ref{thm:tauwords}).

We focus on $(\tau,\bn)$-words in Sections \ref{sec:decomp} and \ref{sec:finalproof} and give a proof for Theorem \ref{thm:tauwords}.
In Section \ref{sec:decomp}, we introduce a height function and a concept of irreducibility for $(\tau, \bn)$-words. Using the height function, we describe an algorithm of finding the unique irreducible initial subword of $(\tau,\bn)$-words of a prescribed height, which provides the main ingredient for proving a decomposition result on words in $S_\tau(\bn, \t)$ and finishing the proof of Theorem \ref{thm:main0} in Section \ref{sec:finalproof}.


\itm In Section \ref{sec:examples}, we give {\it examples} of how to compute the linear function described in Theorems \ref{thm:main0intro} and \ref{thm:main0} through $(\tau, \bn)$-words. In particular, we provide an explicit formula (in Lemma \ref{lem:m=1}) for the linear function when the long-edge graph only has one type of edges. 
\end{ilist}

\subsection*{A note to readers} Part (ii) and part (iii) are completely independent from one another. The readers should be able to understand each part without reading the other.

\subsection*{Acknowledgement} I would like to thank Brian Osserman for suggesting this problem to me and Florian Block for helpful conversations.

\section{Severi degrees via long-edge graphs}\label{sec:vialongedge}

In this section, we state a more detailed version of Theorem \ref{thm:main0intro} (Theorem \ref{thm:main0}) and an important result on the function $\Phi_\bbeta^s$ (Lemma \ref{lem:vanish0}) preceded by all relevant definitions.
We then state the Fomin-Mikhalkin's formula for computing Severi degree $N^{d, \delta}$ using long-edge graphs. Taking the logarithm of the generating function of their formula, we give a formula for $Q^{d, \delta}$ involving templates.

We start with some basic combinatorial definitions and notation that will be used in the rest of the paper. $\N = \{0, 1, 2,\dots,\}$ is the set of nonnegative integers and $\P=\{1, 2, 3,\dots,\}$ is the set of positive integers. Given a positive integer $\ell,$ we denote by $[\ell]$ the set $\{1, 2,\dots, \ell\}.$

\begin{defn}
 Given a long-edge graph $G = (V, E)$ equipped with weight function $\rho,$ we define the {\it multiplicity of $G$} to be 
\[ \mu(G) = \prod_{e \in E} (\rho(e))^2,\]
and the {\it cogenus of $G$} to be
\[ \delta(G) = \sum_{e \in E} \left( l(e) \rho(e) - 1 \right),\]
where for any $e = \{i,j\} \in E$ with $i < j,$ we define $l(e) = j-i.$
Note that any non-empty long-edge graph has positive cogenus.

We define $\minv(G)$ (respectively, $\maxv(G)$) to be the smallest (respectively, largest) vertex of $G$ that has nonzero-degree. We then define the {\it length} of $G,$ denoted by $l(G),$ to be $\maxv(G) -\minv(G).$

For any long-edge graph $G$ and any $k \in \N,$ we denote by $G_{(k)}$ the graph obtained by shifting all edges of $G$ to the right $k$ units, i.e., a weighted edge $\{i,j\}$ in $G$ becomes a weighted edge $\{i+k, j+k\}$ in $G_{(k)}.$
\end{defn}

\begin{ex}
  Consider $G_1$ and $G_2$ in Figure \ref{fig:exlongedge}. One notices that graph $G_2$ is obtained by shifting graph $G_1.$ More precisely, $G_2 = (G_1)_{(3)}.$
  It is clear that
  \[ \mu(G_1) = \mu(G_2) = 2^2 \cdot 1^1 = 4, \quad \delta(G_1)=\delta(G_2) = (2 \cdot 2 - 1) + (1 \cdot 2 - 1) = 2.\]
Also,
\[ \minv(G_1) = 0, \quad \maxv(G_1) = 2, \qquad \minv(G_2) = 3, \quad \maxv(G_2) = 5.\]
Therefore, 
\[ l(G_1) = l(G_2) = 2.\]
\end{ex}

\begin{defn}
	Given a long-edge graph $G,$ we say a tuple $(G_1, \dots, G_i)$ of (non-empty) long-edge graphs is a {\it partition} of $G$ if the disjoint union of the (weighted) edge sets of $G_1, \dots, G_i$ is the (weighted) edge set of $G.$
\end{defn}
By the definitions of multiplicity and cogenus, one checks that for any partition $(G_1, \dots, G_i)$ of $G$, we have
\begin{equation}\label{equ:partofG}
	\mu(G) = \prod_{j=1}^i \mu(G_j) \quad \text{and} \quad \delta(G) = \sum_{j=1}^i \delta(G_j).
\end{equation}

An important family of long-edge graphs is {\it templates}.
\begin{defn}
	A long-edge graph $\Gamma$ is a {\it template} if for any vertex $i:$ $1 \le i \le \maxv(\Gamma)-1$, there exists at least one edge $\{j,k\}$ satisfying $j < i < k.$

We say a long-edge graph $G$ is a {\it shifted template}  if $G$ can be obtained by shifting a template; that is, if $G = \Gamma_{(k)}$ for some template $\Gamma$ and some nonnegative integer $k.$
\end{defn}

It is clear from the above definition that for any template $\Gamma$ and any $k \in \N,$ we have
\begin{equation}\label{equ:shifttempdata}
	\minv\left( \Gamma_{(k)} \right) = k \text{ and } \maxv\left( \Gamma_{(k)} \right) = k + l(\Gamma).
\end{equation}

\begin{ex}
	Consider the three graphs in Figure \ref{fig:exlongedge}. The graph $G_1$ is a template (and also a shifted template), the graph $G_2$ is a shifted template, and the graph $G_3$ is not a shifted template. 
\end{ex}

\begin{defn}
	Let $G$ be a long-edge graph with associated weight function $\rho$. We define
	\[ \lambda_j(G) = \text{ sum of the weights of edges $\{i, k\}$ with $i < j \le k$}, \quad \forall j.\]
	Let $\bbeta= (\beta_1, \beta_2, \dots, \beta_{M+1}) \in \N^{M+1}$ (where $M \ge 0$).
	We say $G$ is {\it $\bbeta$-allowable} if $\maxv(G) \le M+1$ and $\beta_j \ge \lambda_j(G)$ for each $j.$

	A long-edge graph $G$ is {\it strictly $\bbeta$-allowable} if it satisfies the following conditions:
	\begin{alist}
	\itm $G$ is $\bbeta$-allowable.
	\itm Any edge that is incident to the vertex $0$ has weight $1.$ 
	\itm Any edge that is incident to the vertex $M+1$ has weight $1.$ 
	\end{alist}



\end{defn}

\begin{defn}
	Suppose $\bbeta= (\beta_1, \beta_2, \dots, \beta_{M+1}) \in \N^{M+1}$ and $G$ is $\bbeta$-allowable. We create a new graph $\ext_\bbeta(G)$ by adding $\beta_j-\lambda_j(G)$ (unweighted) edges connecting vertices $j-1$ and $j$ for each $1 \le j \le M+1.$ 

An {\it $\bbeta$-extended ordering} of $G$ is a total ordering of the vertices and edges of $\ext_\bbeta(G)$ satisfying the following conditions:
\begin{alist}
\itm The ordering extends the natural ordering of the vertices $0, 1, 2, \cdots$ of $\ext_\bbeta(G)$.
\itm For any edge $e =\{a, b\}$,  
its position in the total ordering has to be between $a$ and $b.$
\end{alist}

We consider two $\bbeta$-extended orderings $o$ and $o'$ to be equivalent if there is an automorphism $\sigma$ on the edges of $\ext_\bbeta(G)$ such that
\begin{alist}
\itm If $\sigma(e) = e',$ then $e$ and $e'$ have the same vertices, and either have the same weights or are both unweighted.
\itm When applying $\sigma$ on the ordering $o$, one obtains the ordering $o'.$
\end{alist}
For any long-edge graph $G$, we define
\[ P_\bbeta(G) = \text{ the number of $\bbeta$-extended orderings (up to equivalence) of $G$},\]
where by convention $P_\bbeta(G) = 0$ if $G$ is not $\bbeta$-allowable, and then define
\[ P_{\bbeta}^s(G) = \begin{cases}P_{\bbeta}(G) & \text{if $G$ is strictly $\bbeta$-allowable;} \\
	0 & \text{otherwise.}
\end{cases}\]
\end{defn}

\begin{rem}\label{rem:emptygraph}
  Suppose $G$ is the empty long-edge graph, i.e., graph without any edges. Then for any $\bbeta \in \N^{M+1},$ the graph $G$ is (strictly) $\bbeta$-allowable and $P_\bbeta(G) = P_\bbeta^s(G) = 1.$
\end{rem}

\begin{ex}Let $\bbeta = (\beta_1, \beta_2, \dots, \beta_{M+1}) \in \N^{M+1}.$ 

  Consider $G_1$ in Figure \ref{fig:exlongedge}. We have
  \[ \lambda_1(G_1) = 3, \quad \lambda_2(G_1) = 1, \quad \text{ and for any $j \ge 3$}, \quad \lambda_j(G_1) = 0.\]
Hence, $G_1$ is $\bbeta$-allowable if and only if $M \ge 1$, $\beta_1 \ge 3$ and $\beta_2 \ge 1,$ but is never strictly $\bbeta$-allowable.

  Consider $G_2$ in Figure \ref{fig:exlongedge}. We have
  \[ \lambda_4(G_2) = 3, \quad \lambda_5(G_2) = 1, \quad \text{ and for any $j \neq 4, 5$}, \quad \lambda_j(G_2) = 0.\]
  One can check that $G_2$ is $\bbeta$-allowable and strictly $\bbeta$-allowable if and only if $M \ge 4,$ $\beta_4 \ge 3$ and $\beta_5 \ge 1.$

  Finally, one checks that $G_3$ in Figure \ref{fig:exlongedge} is $\bbeta$-allowable if and only if $M \ge 5,$ $\beta_4 \ge 3$, $\beta_5 \ge 1$ and $\beta_6 \ge 2;$
  but is strictly $\bbeta$-allowable if and only if $M \ge 6,$ $\beta_4 \ge 3$, $\beta_5 \ge 1$ and $\beta_6 \ge 2.$
\end{ex}

One sees from the above example that a long-edge graph is simultaneously $\bbeta$-allowable and strictly $\bbeta$-allowable most of the time; these two allowabilities only differ on some special ``boundary'' conditions. Therefore, we can focus more on the statistic $P_\bbeta(G)$, although the number $P_\bbeta^s(G)$ will be used to compute the Severi degrees.

\begin{ex}[Examples of $P_\bbeta$] \label{ex:funcP}
Suppose $\bbeta = (\beta_1, \dots, \beta_{M+1}) \in \N^{M+1}$. 
  
  Let $G$ be the long-edge graph with only two edges of weight $2$ connecting vertices $0$ and $1$. 
  Then $G$ is $\bbeta$-allowable if and only if $\beta_1 \ge 4.$
  Thus, $P_\bbeta(G) = 0$ for $\beta_1 < 4.$ For $\beta_1 \ge 4,$ in order to create $\ext_{\bbeta}(G),$ we need to add $\beta_1-4$ unweighted edges connecting vertices $0$ and $1.$ The number of $\bbeta$-extended orderings of $G$ only depends on how we order these $\beta_1-4$ new edges and the original two edges in $G.$ It is easy to see that $P_\bbeta(G) = \binom{\beta_1-4+2}{2} = \binom{\beta_1-2}{2}$ for $\beta_1 \ge 4.$

Let $G'$ be the long-edge graph with only one edge of weight $2$ connecting vertices $0$ and $1$. By a similar discussion, we get that $P_\bbeta( G')$ is $\beta_1-1$ for $\beta_1 \ge 2,$ and is $0$ for $\beta_1 < 2.$
\end{ex}




Although a graph $G$ is $\bbeta$-allowable if $\bbeta \ge \lambda(G),$ it turns out that one only need a weaker version of the condition $\bbeta \ge \lambda(G)$ for our main result. 

\begin{defn}\label{defn:redlambda}
	Let $G$ be a long-edge graph with associated weight function $\rho$. We define
	\[ \olambda_j(G) = \lambda_j(G) - \# (\text{edges in $G$ connecting vertices $j-1$ and $j$}) , \quad \forall j.\]

\end{defn}

Recall definitions of $\Phi_\bbeta(G)$ and $\Phi_\beta^s(G)$ given in the introduction:
\begin{equation}\label{equ:qG}
\Phi_\bbeta(G) := \sum_{i \ge 1} \frac{(-1)^{i+1}}{i} \sum_{(G_1, \dots, G_i)}  \left(\prod_{j=1}^i P_\bbeta(G_j) \right),
\end{equation}
\begin{equation}\label{equ:qdG}
\Phi_\bbeta^s(G) := \sum_{i \ge 1} \frac{(-1)^{i+1}}{i} \sum_{(G_1, \dots, G_i)}  \left(\prod_{j=1}^i P_\bbeta^s(G_j) \right),
\end{equation}
where both summations are over all the partitions of $G$.


We now state a more specific version of our main result Theorem \ref{thm:main0intro}. 

\begin{thm}\label{thm:main0}
	Suppose $G$ is a long-edge graph satisfying $\maxv(G) \le M+1$. Then for any $\bbeta = (\beta_1, \dots, \beta_{M+1})$ satisfying $\beta_j \ge \olambda_j(G)$ for all $j,$ the values $\Phi_\bbeta \left( G \right)$ are given by a linear multivariate function in $\bbeta$. \end{thm}

      \begin{rem}\label{rem:betalength}
	By the definitions of $\Phi_\bbeta(G)$, one sees that $\Phi_\bbeta(G)$ is only determined by numbers $\beta_{\minv(G)+1}, \beta_{\minv(G)+2}, \dots, \beta_{\maxv(G)}$. Hence, the conclusion of Theorem \ref{thm:main0} can be strengthened to ``the values $\Phi_\bbeta \left( G \right)$ are given by a linear multivariate function in $\beta_{\minv(G)+1},$ $\beta_{\minv(G)+2},$ $\dots,$ $\beta_{\maxv(G)}$''.
\end{rem}

\begin{ex}[Example of $\Phi_\bbeta$] \label{ex:funcPhi}
  Let $\bbeta,$ $G$ and $G'$ be as described in Example \ref{ex:funcP}. There are two partitions of $G$: $\left(G\right)$ and $\left(G', G'\right).$ Hence,\[ \Phi_\bbeta\left(G\right) = \frac{(-1)^{1+1}}{1} P_\bbeta\left(G\right)  + \frac{(-1)^{2+1}}{2} P_\bbeta\left(G'\right) P_\bbeta\left(G'\right) = P_\bbeta\left(G\right)  - \frac{1}{2} P_\bbeta^2\left(G'\right).\]
  Recall that we have computed $P_\bbeta\left(G\right)$ and $P_\bbeta(G')$ in Example \ref{ex:funcP}. Hence,
  \[    \Phi_\bbeta\left(G\right) =\begin{cases}
    0 - \frac{1}{2} \cdot 0^2 = 0, & \beta_1 = 0, 1; \\
    0 - \frac{1}{2} \cdot 1^2 = - \frac{1}{2}, & \beta_1 = 2; \\
    0 - \frac{1}{2} \cdot 2^2 = - 2, & \beta_1 = 3; \\
    \binom{\beta_1-2}{2} - \frac{1}{2} \cdot (\beta_1-1)^2 = - \frac{1}{2}(3\beta_1-5), & \beta_1 \ge 4.
  \end{cases}
  \]
  Note that $\olambda(G) = (2, 0, 0, \dots).$ Hence, $\beta_j \ge \olambda_j(G)$ for all $j$ if and only if $\beta_1 \ge 2.$ We check that for $\beta_1= 2,$ $\Phi_\bbeta(G) = -\frac{1}{2} = \left.- \frac{1}{2}(3\beta_1-5)\right|_{\beta_1=2}$ and, for $\beta_1= 3,$ $\Phi_\bbeta(G) = -2 = \left.- \frac{1}{2}(3\beta_1-5)\right|_{\beta_1=3}$. Hence, $\Phi_\bbeta(G) =  - \frac{1}{2}(3\beta_1-5),$ for all $\beta_1 \ge 2,$ agreeing with Theorem \ref{thm:main0}.
\end{ex}
We see from the above example that unlike $P_\bbeta(G),$ the value of $\Phi_\bbeta(G)$ is not necessarily $0$ when $G$ is not $\bbeta$-allowable. (In fact, $\Phi_\bbeta(G)$ is not even necessarily $0$ when $G$ does not satisfy the condition $\beta_j \ge \olambda_j(G)$ for all $j.$)

In addition to Theorem \ref{thm:main0}, which is a result on $\Phi_\bbeta(G),$ we also have a fundamental but important result on $\Phi_\bbeta^s(G).$ 
\begin{lem}\label{lem:vanish0}
Suppose $G$ is not a shifted template. Then
\[ \Phi_\bbeta^s(G) = 0.\]
\end{lem}


Theorem \ref{thm:main0} and Lemma \ref{lem:vanish0} will be reformulated in Section \ref{sec:reformulate} and proved afterwards. Before that, we discuss the results we obtain by applying them to the classical Severi degrees. 
For the rest of the section, we will introduce Fomin-Mikhalkin's formula for the classical Severi degree using long-edge graphs and its logarithmic version, and discuss an immediate consequence of Lemma \ref{lem:vanish0} on the logarithmic version of the formula, providing the original motivation for the author to consider the functions $\Phi_\bbeta^s$ and $\Phi_\bbeta$.

\subsection*{Analyzing Fomin-Mikhalkin's formula}
For the classical Severi degree, we only need to use $\bbeta = (0, 1, 2, \dots, d)$. Therefore, we give the following notation:
\begin{align}
  &\v(d) := (0, 1, 2, \dots, d), \qquad  \forall d \in \P.
\end{align}


Below is Fomin-Mikhalkin's formula for classical Severi degrees \cite{FomMik2010}.
\begin{thm}[Fomin-Mikhalkin]
  The Severi degree $N^{d, \delta}$ is given by
  \begin{equation}
    N^{d, \delta} = \sum_G \mu(G) P_{\v(d)}^s(G),
  \end{equation}
  where the summation is over all the long-edge graphs of cogenus $\delta$. 
\end{thm}

Recall that $\cN(d)$ and $\cQ(d)$ are defined as in \eqref{equ:defncN} and \eqref{equ:defncQ} respectively.
Applying the above theorem to the generating function $\cN(d)$, we get
\[ \cN(d) = 1 + \sum_{\delta \ge 1} N^{d, \delta} t^\delta = 1 + \sum_G \mu(G) P_{\v(d)}^s(G) \ t^{\delta(G)},\]
where the summation is over all the (non-empty) long-edge graphs.
Taking logarithms on both sides of the above formula, we obtain
\begin{equation}\label{equ:qddelta0} 
  Q^{d, \delta} = \sum_{i \ge 1} \frac{(-1)^{i+1}}{i} \sum_{(G_1, \dots, G_i)}  \left(\prod_{j=1}^i \mu(G_j) P_{\v(d)}^s(G_j)\right),
\end{equation}
where the summation is over all the tuples $(G_1, \dots, G_i)$ of (non-empty) long-edge graphs satisfying $\sum_{j=1}^i \delta(G_j) = \delta.$ Since we can consider any such tuple a partition of a long-edge graph of cogenus $\delta,$
by \eqref{equ:partofG} and \eqref{equ:qdG}, we can rewrite \eqref{equ:qddelta0}:
\begin{equation}\label{equ:qddelta} 
  Q^{d, \delta} = \sum_G \mu(G) \Phi_{\v(d)}^s(G),
  \end{equation}
  where the summation is over all the long-edge graphs of cogenus $\delta$. This is the reason why we consider $\Phi_\bbeta^s(G)$ (respectively, $\Phi_\bbeta(G)$) the logarithmic version of $P_\bbeta^s(G)$ (respectively, $P_\bbeta(G)$).

One benefit of computing $Q^{d, \delta}$ instead of $N^{d, \delta}$ is that a lot of terms in \eqref{equ:qddelta} vanish. 
%
We have the following corollary to Lemma \ref{lem:vanish0}.

\begin{cor}\label{cor:shifttemp} For $\delta \ge 1,$
  \begin{equation}\label{equ:shifttemp}
    Q^{d, \delta} = \sum_\Gamma \mu(\Gamma) \sum_{k \in \N} \Phi_{\v(d)}^s\left(\Gamma_{(k)}\right),
\end{equation}
  where the first summation is over all the templates of cogenus $\delta$. 
\end{cor}

It is an easy fact that for any fixed $\delta,$ there are finitely many templates of cogenus $\delta.$ Hence, the first summation in \eqref{equ:shifttemp} is finite. 
It is not hard to see that the second summation in \eqref{equ:shifttemp} has finitely many non-zero terms as well. 
Moreover, intuitively the linear function described in Theorem \ref{thm:main0} is important for computing the second summation, therefore potentially lead to combinatorial formulas for $Q^{d,\delta}$ and $Q_\delta(d).$

In the next two sections, we discuss details of results obtained by applying Theorem \ref{thm:main0} to Formula \eqref{equ:shifttemp}. The material presented in Sections \ref{sec:computeQ} and \ref{sec:determine} is irrelevant to the rest of the paper. The reader should feel free to skip it.

\section{On functions $Q^{d, \delta}$ and $Q_\delta(d)$}\label{sec:computeQ}

In this section, assuming Corollary \ref{cor:shifttemp} and Theorem \ref{thm:main0}, we prove Corollary \ref{cor:quadratic}. (Recall that Corollary \ref{cor:quadratic} states that for fixed $\delta,$ the function $Q^{d, \delta}$ is quadratic in $d$ for sufficiently large $d.$) 
We will then demonstrate how one can compute the quadratic polynomial $Q_\delta(d)$ from our results. 

\subsection{Qudraticity of $Q^{d, \delta}$ and $Q_\delta(d)$}

Because of Corollary \ref{cor:shifttemp}, we will mostly focus on templates. For convenience and clearness, we state a version of Theorem \ref{thm:main0} for templates, which follows directly from Theorem \ref{thm:main0} and Remark \ref{rem:betalength}.
\begin{cor}\label{cor:maintemplate}
Suppose $\Gamma$ is a template of length $\ell.$ Then for any $\bbeta=(\beta_1, \dots, \beta_\ell)$ satisfying $\beta_j \ge \olambda_j(\Gamma)$ for all $j,$ the values $\Phi_\bbeta(\Gamma)$ are given by a linear multivariate function in $\bbeta.$
\end{cor}

\begin{defn}
	We denote by $\Phi(\Gamma, \bbeta)$ the linear function described in Corollary \ref{cor:maintemplate}.
\end{defn}

By \eqref{equ:shifttemp}, it is natural to define for each template $\Gamma,$
\begin{equation}\label{equ:defnQdGamma} 
Q^{d, \Gamma} := \mu(\Gamma) \sum_{k \in \N} \Phi_{\v(d)}^s\left(\Gamma_{(k)}\right).
\end{equation}
Hence, it is sufficient to show that $Q^{d, \Gamma}$ is quadratic for sufficiently large $d$. We will prove this by analyzing functions $\Phi_\bbeta^s$ and $\Phi_\bbeta$ further and give a more precise formula for $Q^{d, \Gamma}.$

We start with a preliminary defintion and a lemma.
\begin{defn}\label{defn:epsilon}
Let $G$ be a long-edge graph. We define
 \begin{equation} \epsilon_0(G) = \begin{cases}1, & \text{if all edges adjacent to the vertex $\minv(G)$ have weight $1$;} \\
    0, & \text{otherwise.}\end{cases}
    \end{equation}
 \begin{equation} \epsilon_1(G) = \begin{cases}1, & \text{if all edges adjacent to the vertex $\maxv(G)$ have weight $1$;} \\
    0, & \text{otherwise.}\end{cases}
    \end{equation}
  \end{defn}

\begin{lem}\label{lem:PPhi}
	Suppose $\bbeta = (\beta_1, \beta_2, \dots, \beta_{M+1}) \in \N^{M+1}.$
  \begin{ilist}
  \itm 
  $P_\bbeta^s(G) = P_\bbeta(G)$ if $M \ge \maxv(G)-\epsilon_1(G)$ and $1 \le \minv(G) + \epsilon_0(G)$, and is $0$ otherwise.

  \itm 
  $\Phi_\bbeta^s(G) = \Phi_\bbeta(G)$ if $M \ge \maxv(G)-\epsilon_1(G)$ and $1 \le \minv(G) + \epsilon_0(G)$, and is $0$ otherwise.
\end{ilist}
\end{lem}

\begin{proof}
(i) follows directly from the defintions of $P_\bbeta(G)$ and $P_\bbeta^s(G).$

We use (i) to prove (ii). Comparing equations \eqref{equ:qdG} and \eqref{equ:qG} for $\Phi_\bbeta^s(G)$ and $\Phi_\bbeta(G),$ it is sufficient to prove that for any partition $(G_1, \dots, G_i)$ of $G$, we have
\[ \prod_{j=1}^i P_\bbeta^s(G_j) = \begin{cases}
  \prod_{j=1}^i P_\bbeta(G_j) & \text{ if $M \ge \maxv(G)-\epsilon_1(G)$ and $1 \le \minv(G) + \epsilon_0(G)$}; \\
  0 & \text{ otherwise.}
\end{cases}\]
Suppose $(G_1, \dots, G_i)$ is a partition of $G.$ We check that for any $1 \le j \le i,$ we have 
\[ \maxv(G_j) - \epsilon_1(G_j) \le \maxv(G) - \epsilon_1(G) \text{ and } \minv(G_j) + \epsilon_0(G_j) \ge \minv(G) + \epsilon_0(G),\]
and each equality holds for at least one $j.$ 
Hence, if $M \ge \maxv(G)-\epsilon_1(G)$ and $1 \le \minv(G) + \epsilon_0(G)$, we have $M \ge \maxv(G_j)-\epsilon_1(G_j)$ and $1 \le \minv(G_j) + \epsilon_0(G_j)$ for all $j.$
Thus, by (i), $\prod_{j=1}^i P_\bbeta^s(G_j) = \prod_{j=1}^i P_\bbeta(G_j)$. 

Otherwise, we have $M < \maxv(G)-\epsilon_1(G)$ or $1 < \minv(G) + \epsilon_0(G)$. Then there exists $j$ such that $M < \maxv(G_j)-\epsilon_1(G_j)$ or $1 < \minv(G_j) + \epsilon_0(G_j)$, which implies that $P_\bbeta^s(G_j) = 0$ by (i). Hence, we have $\prod_{j=1}^i P_\bbeta^s(G_j) = 0.$
\end{proof}

\begin{cor}\label{cor:Phitemp}
	Suppose $\bbeta = (\beta_1, \beta_2, \dots, \beta_{M+1}) \in \N^{M+1}$ and $\Gamma$ is a template. Then
  $\Phi_{\bbeta}^s\left( \Gamma_{(k)} \right) = \Phi_{\bbeta}\left( \Gamma_{(k)} \right)$ if 
  $1 - \epsilon_0(\Gamma) \le k \le M + \epsilon_1(\Gamma) - l(\Gamma),$
  and is $0$ otherwise.
\end{cor}
\begin{proof}
	This immediately follows from Lemma \ref{lem:PPhi}/(ii) and  \eqref{equ:shifttempdata}.
\end{proof}

We now use the above corollary to give a precise summation formula for $Q^{d, \Gamma}$ (defined as in \eqref{equ:defnQdGamma}) which only involves values in the form of $\Phi_\bbeta(\Gamma),$
defining
\[\v(k, \ell) := (k, k+1, \dots, k+\ell-1), \qquad \forall k \in \N \text{ and } \forall \ell \in \P.\]
\begin{cor}\label{cor:formQdGamma}
  Suppose $\Gamma$ is a template. Then
  \[ Q^{d,\Gamma}   =\begin{cases}
	  \ds \mu(\Gamma) \sum_{k=1}^{d + \epsilon_1(\Gamma) - l(\Gamma)} \Phi_{\v(k, l(\Gamma))}\left( \Gamma \right),& \text{ if $d \ge l(\Gamma) - \epsilon_1(\Gamma)$};\\
    0, & \text{ if $d < l(\Gamma) - \epsilon_1(\Gamma)$}.
  \end{cases}\]
\end{cor}

\begin{proof}By Corollary \ref{cor:Phitemp} and the fact that $\Gamma_{(0)} = \Gamma$ is not $\v(d)$-allowable, we immediately have that
  \[ Q^{d,\Gamma}   =\begin{cases}
	  \ds \mu(\Gamma) \sum_{k=1}^{d + \epsilon_1(\Gamma) - l(\Gamma)} \Phi_{\v(d)}\left( \Gamma_{(k)} \right),& \text{ if $d \ge l(\Gamma) - \epsilon_1(\Gamma)$};\\
    0, & \text{ if $d < l(\Gamma) - \epsilon_1(\Gamma)$}.
  \end{cases}\]
Moreover, assuming $1 \le k \le d + \epsilon_1(\Gamma)-l(\Gamma),$ it follows from the definition of $\Phi_\bbeta(G)$ that,
\[ \Phi_{\v(d)}\left( \Gamma_{(k)} \right) = \Phi_{\v(k,l(\Gamma))}(\Gamma).\]
\end{proof}

Because of Corollary \ref{cor:formQdGamma}, we want to determine when $\Phi_{\v(k, l(\Gamma))}(\Gamma)$ is given by the linear function described in Corollary \ref{cor:maintemplate}.

Note that\[ \v(k, l(\Gamma)) = (k, k+1, \dots, k+l(\Gamma)-1) \ge (\olambda_1(\Gamma), \olambda_2(\Gamma), \dots, \olambda_{l(\Gamma)}(\Gamma)) \] 
if and only if
\[ k \ge \olambda_j(\Gamma) -j+1, \qquad \forall j.\]
Hence, it's natural to define
		\begin{equation}\label{equ:kmin0}
			k_{\min}(\Gamma) := \max\left( 1, \max \{ \olambda_j(\Gamma)-j+1 \}\right).
		\end{equation}
		(We remark that the definition of $k_{\min}$ is different from those defined in \cite{FomMik2010} or \cite{block2011}.)

Then we have the following result.
%
\begin{lem}\label{lem:linearink}
	Suppose $\Gamma$ is a template of length $\ell$. For any $k \ge k_{\min}(\Gamma)$, the values $\Phi_{\v(k,\ell)}\left( \Gamma \right)$ is a linear function in $k$. 
\end{lem}

\begin{proof}
	As we discussed above that when $k \ge k_{\min}(\Gamma),$ we have 
\[ \v(k, \ell) = (k, k+1, \dots, k+\ell-1) \ge (\olambda_1(\Gamma), \olambda_2(\Gamma), \dots, \olambda_{\ell}(\Gamma)).\]
Hence, by Corollary \ref{cor:maintemplate}, 
\begin{equation}\label{equ:Phibeta2Phik}
\Phi_{\v(k, \ell)}(\Gamma) = \left.\Phi(\Gamma, \bbeta)\right|_{\beta_1=k, \beta_2 = k+1, \dots, \beta_{\ell} = k+\ell-1},
\end{equation}
which clearly is a linear function in $k.$
\end{proof}

We can finally state and prove the quadraticity result of $Q^{d, \Gamma}$.
\begin{cor}\label{cor:quadGamma}
Suppose $\Gamma$ is a template. Then $Q^{d, \Gamma}$ is a quadratic polynomial in $d$ for $d \ge k_{\min}(\Gamma)+l(\Gamma) - \epsilon_1(\Gamma) -1$.

\end{cor}

\begin{proof}
For $y \ge k_{\min}(\Gamma),$ we define 
\[ F(y) := \sum_{k = k_{\min}(\Gamma)}^{y} \Phi_{\v(k,l(\Gamma))}\left( \Gamma \right).\]
By Lemma \ref{lem:linearink},  $F(y)$ is a quadratic polynomial in $y$, thus we can extend the domain of $F(y)$ to $\Z.$ It's a consequence of Faulhaber's formula that $F(k_{\min}-1) = 0.$

By Corollary \ref{cor:formQdGamma}, if $d \ge k_{\min}(\Gamma)+l(\Gamma) - \epsilon_1(\Gamma)-1,$ we have
  \begin{equation} \label{equ:QdGamma}  
  Q^{d, \Gamma} = \mu(\Gamma) \left(\sum_{k =1}^{k_{\min}(\Gamma)-1} \Phi_{\v(k, l(\Gamma))}\left( \Gamma \right) + \left.F(y)\right|_{y= d+\epsilon_1(\Gamma)-l(\Gamma)} \right). 
\end{equation}
Clearly the first summation is a constant, and second term is a quadratic polynomial in $d$.
\end{proof}


\begin{rem}\label{rem:threshold}
  Corollary \ref{cor:quadGamma} not only implies Corollary \ref{cor:quadratic}, but also tells us that $Q^{d, \delta}$ is quadratic in $d$ and $N^{d, \delta}$ is polynomial in $d$, for $d$ satisfying $d \ge k_{\min}(\Gamma)+l(\Gamma) - \epsilon_1(\Gamma)-1$ for each template $\Gamma$ of cogenus $\delta.$ 
  
  With some analysis of bounds for ${\olambda}_j,$ e.g. Lemma 4.2 of \cite{toricsurface}, and the easy observation $l(\Gamma)-\epsilon_1(\Gamma) \le \delta,$ one can show that $k_{\min}(\Gamma)+l(\Gamma) - \epsilon_1(\Gamma)-1$ is always no greater than $\delta.$ Therefore, we recover the threshold bound $d^* \le \delta$ obtained by Block \cite{block2011}.
\end{rem}

\subsection{Computing $Q_\delta(d)$}
We denote by $Q_\Gamma(d)$ the quadratic polynomial correpsonding to $Q^{d, \Gamma}.$ 
Therefore,
we have
\begin{align}
  Q_\delta(d) = & \sum_{\Gamma} Q_\Gamma(d). \label{equ:QdeltaQGamma}
\end{align}
where the summation is over all the templates of cogenus $\delta.$ In order to compute the polynomial $Q_\delta(d),$ we need to compute $Q_\Gamma(d)$ for each template $\Gamma$ of cogenus $\delta.$ Table \ref{tab:templates} lists all the templates of cogenus $1$ or $2$. 
Note that since every template $\Gamma$ has at least one edge incident to the vertex $0,$ we omit labels of vertices when drawing a template $\Gamma$ and assume that the vertices are $0, 1, \dots, l(\Gamma).$

\begin{sidewaystable}
\centering
\vspace*{15cm}
\begin{tabular}{c|c|c|c|c|c|c|c|c|c|c|c|c|c|c|}
$\Gamma$ &
$\hspace*{-1mm} \delta(\Gamma) \hspace*{-1mm}$ & $\hspace*{-1mm} l(\Gamma) \hspace*{-1mm}$ 
& $\hspace*{-1mm} \mu(\Gamma) \hspace*{-1mm}$ 
& $\hspace*{-1mm} \epsilon_0(\Gamma) \hspace*{-1mm}$ & $\hspace*{-1mm} \epsilon_1(\Gamma) \hspace*{-1mm}$ & $\lambda(\Gamma)$ & $\olambda(\Gamma)$ 

& \hspace*{-2mm} $k_{\min}(\Gamma)$ \hspace*{-2mm}
& $P(\Gamma, \bbeta)$
& $\Phi(\Gamma, \bbeta)$
& $\hspace*{-1mm} \zeta^0(\Gamma) \hspace*{-1mm}$ & $\hspace*{-1mm} \zeta^1(\Gamma) \hspace*{-1mm}$ & $\hspace*{-1mm} \eta_0(\Gamma) \hspace*{-1mm}$
& \hspace*{-1mm} \begin{tabular}{c}$\scriptstyle \Phi_{(k, l(\Gamma))}\left(\Gamma\right)$ \\ $\scriptstyle k \ge k_{\min}(\Gamma)$\end{tabular} \hspace*{-1mm}
\\
\hline \hline
&&&&&&&&&&&&&&\\[-.1in]
\begin{picture}(70,8)(-5,-4)\setlength{\unitlength}{2.5pt}\thicklines
\oo
\put(5,2){\makebox(0,0){$\scriptstyle 2$}}
\Eeee
\end{picture}
& 1 & 1 & 4 & 0 & 0 & (2) & (1) & 1 & $\beta_1-1$ & $\beta_1-1$ 
& 1 & 0 & -1 & $k-1$
\\[.15in]
\begin{picture}(70,8)(-5,-4)\setlength{\unitlength}{2.5pt}\thicklines
\ooo
\qbezier(0.8,0.6)(10,5)(19.2,0.6)
\put(10,3.5){\makebox(0,0){$\scriptstyle 1$}}
\end{picture}
& 1 & 2 & 1 & 1 & 1 & (1,1) & (1,1) & 1 & $\beta_1+\beta_2$ & $\beta_1+\beta_2$
& 2 & 1 & 0 & $2k+1$
\\
&&&&&&&&&&&&&&\\[-.1in]
\hline \hline
&&&&&&&&&&&&&&\\[-.1in]
\begin{picture}(70,8)(-5,-4)\setlength{\unitlength}{2.5pt}\thicklines
\oo
\put(5,2){\makebox(0,0){$\scriptstyle 3$}}
\Eeee
\end{picture}
& 2 & 1 & 9 & 0 & 0 & (3) & (2) & 2 & $\beta_1-2$ & $\beta_1-2$ 
& 1 & 0 & -2 & $k-2$
\\[.15in]
\begin{picture}(70,8)(-5,-4)\setlength{\unitlength}{2.5pt}\thicklines
\oo
\put(5,3.5){\makebox(0,0){$\scriptstyle 2$}}
\put(5,-3.5){\makebox(0,0){$\scriptstyle 2$}}
\qbezier(0.8,0.6)(5,2)(9.2,0.6)
\qbezier(0.8,-0.6)(5,-2)(9.2,-0.6)
\end{picture}
& 2 & 1 & 16 & 0 & 0 & (4) & (2) & 2 & $\binom{\beta_1-2}{2}$ & $-\frac{3}{2}\beta_1+\frac{5}{2}$ 
& -$\frac{3}{2}$ & 0 & $\frac{5}{2}$ & $-\frac{3}{2}k+\frac{5}{2}$
\\[.15in]
\begin{picture}(70,8)(-5,-4)\setlength{\unitlength}{2.5pt}\thicklines
\ooo
\qbezier(0.8,0.6)(10,4)(19.2,0.6)
\qbezier(0.8,-0.6)(10,-4)(19.2,-0.6)
\put(10,3.5){\makebox(0,0){$\scriptstyle 1$}}
\put(10,-3.5){\makebox(0,0){$\scriptstyle 1$}}
\end{picture}
& 2 & 2 & 1 & 1 & 1 & (2,2) & (2,2) & 2 & $\binom{\beta_1+\beta_2-1}{2}$ & $\scriptstyle -\frac{3}{2}\beta_1-\frac{3}{2}\beta_2+1$ 
& -3 & -$\frac{3}{2}$ & 1 & $-3k - \frac{1}{2}$
\\[.15in]
\begin{picture}(70,8)(-5,-4)\setlength{\unitlength}{2.5pt}\thicklines
\ooo
\qbezier(0.8,0.6)(10,4)(19.2,0.6)
\put(5,-2){\makebox(0,0){$\scriptstyle 2$}}
\Eeee
\put(10,3.5){\makebox(0,0){$\scriptstyle 1$}}
\end{picture}
& 2 & 2 & 4 & 0 & 1 & (3,1) & (2,1) & 2 & $\scriptstyle (\beta_1-2)(\beta_1+\beta_2+1)$ & $\scriptstyle -2\beta_1-\beta_2+2$
& -3 & -1 & 2 & $-3k+1$
\\[.15in]
\begin{picture}(70,8)(-5,-4)\setlength{\unitlength}{2.5pt}\thicklines
\ooo
\qbezier(0.8,0.6)(10,4)(19.2,0.6)
\put(15,-2){\makebox(0,0){$\scriptstyle 2$}}
\eEee
\put(10,3.5){\makebox(0,0){$\scriptstyle 1$}}
\end{picture}
& 2 & 2 & 4 & 1 & 0 & (1,3) & (1,2) & 1 & $\scriptstyle (\beta_2-2)(\beta_1+\beta_2+1)$ & $\scriptstyle -\beta_1-2\beta_2+2$
& -3 & -2 & 2 & $-3k$
\\[.15in]
\begin{picture}(70,8)(-5,-4)\setlength{\unitlength}{2pt}\thicklines
\oooo
\qbezier(0.8,0.6)(15,6)(29.2,0.6)
\put(15,4.5){\makebox(0,0){$\scriptstyle 1$}}
\end{picture}
& 2 & 3 & 1 & 1 & 1 & $\scriptstyle (1,1,1)$ & $\scriptstyle (1,1,1)$ & 1 & $\scriptstyle \beta_1+\beta_2+\beta_3$ & $\scriptstyle \beta_1+\beta_2+\beta_3$ 
& 3 & 3 & 0 & $3k+3$
\\[.15in]
\begin{picture}(70,8)(-5,-4)\setlength{\unitlength}{2pt}\thicklines
\oooo
\qbezier(0.8,0.6)(10,5)(19.2,0.6)
\qbezier(10.8,0.6)(20,5)(29.2,0.6)
\put(9,4){\makebox(0,0){$\scriptstyle 1$}}
\put(21,4){\makebox(0,0){$\scriptstyle 1$}}
\end{picture}
& 2 & 3 & 1 & 1 & 1 & $\scriptstyle (1,2,1)$ & $\scriptstyle (1,2,1)$ & 1 & $\begin{matrix}\scriptstyle \beta_1\beta_3 + \\ \scriptstyle (\beta_2-1)(\beta_1+\beta_2+\beta_3)\end{matrix}$ & $\scriptstyle -\beta_1-\beta_2-\beta_3$
& -3 & -3 & 0 & $-3k-3$
\\[-.05in]
&&&&&&&&&&&&&&\\[-.1in]
\end{tabular}
\bigskip
\bigskip
\caption{The templates with $\delta(\Gamma) \le 2$.}
\label{tab:templates}
\end{sidewaystable}

Below we first give an example of computing $Q_\Gamma(d),$ and then compute $Q_\delta(d)$ for $\delta=1$ and $2.$

\begin{ex}\label{ex:QGamma}
	Let $\Gamma$ be the second template of cogenus $2$ in Table \ref{tab:templates}. Note that $\Gamma$ is the same long-edge graph $G$ we considered in Examples \ref{ex:funcP} and \ref{ex:funcPhi}. When $d$ is sufficiently large, $Q_\Gamma(d) = Q^{d, \Gamma}$ is defined by \eqref{equ:QdGamma}. Since $l(\Gamma) = 1,$ we have \[\Phi_{\v(k, l(\Gamma))}(\Gamma) = \left.\Phi_\bbeta(\Gamma)\right|_{\beta_1=k}.\]
  Hence, using the results for $\Phi_\bbeta\left( \Gamma \right)$ given in Example \ref{ex:funcPhi} and the data in Table \ref{tab:templates}, we get
  \[ Q_\Gamma(d) = 16 \cdot \left(0 + \sum_{k=2}^{d+0-1} -\frac{1}{2}(3k-5) \right) = -12d^2 + 52d - 56.\]
\end{ex}

\begin{ex}
  Let $\delta = 1.$ There are two templates of cogenus $1$ as listed in Table \ref{tab:templates}. We denote them by $\Gamma_1$ and $\Gamma_2$ in the order as listed in the Table. We compute $Q_{\Gamma_1}(d)$ and $Q_{\Gamma_2}(d)$ similarly as shown in Example \ref{ex:QGamma} but omit details of how we obtain $\Phi_{\v(k, l(\Gamma_i))}(\Gamma_i)$ for $k < k_{\min}(\Gamma_i).$

\begin{align*}
  Q_{\Gamma_1}(d) =& 4 \cdot \left( \sum_{k=1}^{d+0-1} (k-1)  \right)= 2d^2-6d+4, \\
  Q_{\Gamma_2}(d) 
  =& 1 \cdot \left(\sum_{k=1}^{d+1-2} (2k+1) \right) = d^2-1. 
\end{align*}
Therefore,
\[  Q_1(d) = Q_{\Gamma_1}(d)+Q_{\Gamma_2}(d) 
  = 3d^2 - 6d +3 = 3(d-1)^2.\]
\end{ex}

\begin{ex}
Let $\delta = 2.$ There are seven templates of cogenus $2$ as listed in Table \ref{tab:templates}. We denote them by $\Gamma_1, \Gamma_2, \dots, \Gamma_6,$ and $\Gamma_7$ in the order as listed in the Table.
As in the previous example, we compute $Q_{\Gamma_i}(d)$ without details of how we obtain $\Phi_{\v(k, l(\Gamma_i))}(\Gamma_i)$ for $k < k_{\min}(\Gamma_i).$

\begin{align*}
  Q_{\Gamma_1}(d) =& 9 \cdot \left(0  + \sum_{k=2}^{d+0-1} (k-2)\right) = \frac{1}{2}(9d^2-45d+54),\\
Q_{\Gamma_2}(d)  =& 16 \cdot \left(0 + \sum_{k=2}^{d+0-1} -\frac{1}{2}(3k-5) \right) = -12d^2 + 52d - 56, \\
Q_{\Gamma_3}(d)  =& 1 \cdot \left(- \frac{9}{2} + \sum_{k=2}^{d+1-2} -\frac{1}{2}(6k+1) \right) = -\frac{1}{2}(3d^2-2d+1), \\
Q_{\Gamma_4}(d)  =& 4 \cdot \left( 0 + \sum_{k=2}^{d+1-2} (-3k+1) \right) = -6d^2+10d+4, \\
Q_{\Gamma_5}(d) =& 4 \cdot \left( \sum_{k=1}^{d+0-2} (-3k) \right) = -6d^2+18d-12. 
\end{align*}
Finally, by comparing formulas
\[ Q_{\Gamma_6}(d)  = 1 \cdot \left(  \sum_{k=1}^{d+1-3} (3k+3) \right) \quad \text{ and } \quad Q_{\Gamma_7}(d)  = 1 \cdot \left(  \sum_{k=1}^{d+1-3} (-3k-3) \right), \]
we see that $Q_{\Gamma_6}(d) + Q_{\Gamma_7}(d)=0$ without calculation. Therefore,
\[ Q_2(d) = Q_{\Gamma_1}(d) + Q_{\Gamma_2}(d) + Q_{\Gamma_3}(d) + Q_{\Gamma_4}(d) + Q_{\Gamma_5}(d) =-\frac{1}{2}(42d^2-117d+75). \]
\end{ex}

\section{Determining $A_1(t)$ and $A_2(t)$}\label{sec:determine}

In this section, we show how our formula \eqref{equ:QdGamma} for $Q^{d,\Gamma}$ leads to combinatorial formulas for $A_1(t)$ and $A_2(t)$ of Proposition \ref{equ:T2N}. 
For brevity, throughout this section, unless otherwise specified, we always assume $\delta$ is fixed and use $\displaystyle \sum_{\Gamma}$ to denote the summation of all the templates of cogneus $\delta.$ For any power series $F(x),$ we denote by $[x^\delta] F(x)$ the coefficient of $x^\delta$ in $F(x).$

Applying \eqref{equ:T2N} to Proposition \ref{prop:logT}, we get
 \begin{equation}\label{equ:genQdeltad}  \sum_{\delta \ge 1} Q_\delta(d) t^\delta = \log \left( \sum_{\delta \ge 0} N_\delta(d) t^\delta \right) = d^2 A_1(t) - 3d A_2(t) + 9 A_3(t) + 3 A_4(t).
 \end{equation}
Therefore, it suffices to figure out the coefficients of $d^2$ and $d$ in $Q_\delta(d)$. It turns out that we only care about three linear combinations of the coefficients in $\Phi(\Gamma, \bbeta)$ for answering this question.

\begin{defn}
Suppose $\Gamma$ is a template of length $\ell,$ and $\Phi(\Gamma, \bbeta) = a_0 + \sum_{j=1}^{\ell} a_j \beta_j.$ We define
\[ \zeta^i(\Gamma) := \sum_{j=1}^\ell \binom{j-1}{i} a_j \quad \text{ for $i=0,1$,}   \qquad \text{ and } \qquad \eta_0(\Gamma) := a_0.\]
\end{defn}

See Table \ref{tab:templates} for $\zeta^0(\Gamma)$, $\zeta^1(\Gamma)$ and $\eta_0(\Gamma)$ of templates of cognues $\le 2.$

\begin{prop}\label{prop:coeff}
For any $\delta \ge 1,$ we have 
\begin{align}
  \left[ d^2 \right] Q_\delta(d) =& \frac{1}{2} \sum_\Gamma \mu(\Gamma) \zeta^0(\Gamma), \nonumber \\
\left[ d \right] Q_\delta(d)=& \label{equ:coeffd0}
\sum_\Gamma \mu(\Gamma) \left( \frac{1}{2}(2 \epsilon_1(\Gamma) - 2 l(\Gamma) +1) \zeta^0(\Gamma) + \zeta^1(\Gamma) + \eta_0(\Gamma) \right).
\end{align}
\end{prop}

\begin{proof}
  Suppose $\Gamma$ is a template and $\Phi(\Gamma, \bbeta) = a_0 + \sum_{j=1}^{\ell} a_j \beta_j.$ Then for $k \ge k_{\min}(\Gamma),$ applying \eqref{equ:Phibeta2Phik}, we get 
\[\Phi_{\v(k, l(\Gamma))}\left( \Gamma \right) = \left(\sum_{j=1}^\ell a_j\right) k + \sum_{j=1}^\ell (j-1) a_j + a_0 = \zeta^0(\Gamma) k + \zeta^1(\Gamma) + \eta_0(\Gamma). \]
Plugging in \eqref{equ:QdGamma}, we get 
\begin{align*} Q^{d, \Gamma} =& \mu(\Gamma) \left(\sum_{k =1}^{k_{\min}(\Gamma)-1} \Phi_{\v(k, l(\Gamma))}\left( \Gamma \right) + \sum_{k = k_{\min}(\Gamma)}^{d + \epsilon_1(\Gamma) - l(\Gamma)} \left(\zeta^0(\Gamma) k + \zeta^1(\Gamma) + \eta_0(\Gamma)\right) \right) \\
=& \mu(\Gamma) \left(\sum_{k =1}^{k_{\min}(\Gamma)-1} \left(\Phi_{\v(k, l(\Gamma))}\left( \Gamma \right) -(\zeta^0(\Gamma) k + \zeta^1(\Gamma) + \eta_0(\Gamma)) \right) + \sum_{k = 1}^{d + \epsilon_1(\Gamma) - l(\Gamma)} (\zeta^0(\Gamma) k + \zeta^1(\Gamma) + \eta_0(\Gamma)) \right) 
\end{align*}
Since for $d \ge k_{\min}(\Gamma) + l(\Gamma) -\epsilon_1(\Gamma)-1,$ the polynomial $Q_\Gamma(d)$ is computed by the above formula, the first summation of which is a constant and the second summation of which gives a quadratic polynomial of $d,$ we see that the coefficients of $d$ ($d^2$, respectively) of $Q_\Gamma(d)$ is the same as the coefficients of $d$ ($d^2$, respectively) of
\begin{align*}
& \mu(\Gamma) \sum_{k = 1}^{d + \epsilon_1(\Gamma) - l(\Gamma)} (\zeta^0(\Gamma) k + \zeta^1(\Gamma) + \eta_0(\Gamma)) \\
=& \mu(\Gamma) \left( \zeta^0(\Gamma) \frac{(d + \epsilon_1(\Gamma) - l(\Gamma))(d + \epsilon_1(\Gamma) - l(\Gamma)+1)}{2} + (\zeta^1(\Gamma) + \eta_0(\Gamma))(d + \epsilon_1(\Gamma) - l(\Gamma)) \right).
\end{align*}
Then the conclusion follows from \eqref{equ:QdeltaQGamma} and rearranging the terms in the above formula.
\end{proof}

We can simplify the formula \eqref{equ:coeffd0} for $[d] Q_\delta(d)$ using the concept of {\it conjugation}.

\begin{defn}
The {\it conjugate} of a template $\Gamma$, denoted by $\nBar{\Gamma},$ is the template obtained by flipping $\Gamma$ and renaming the vertices accordingly. E.g., in Table \ref{tab:templates}, the $4$th and $5$th templates of cogenus $2$ are conjuate to each other, and the rest of the templates are self-conjugate.
\end{defn}

\begin{lem}\label{lem:conjugate}
Suppose $\Gamma$ is a template. Then
\begin{align} \mu(\nBar{\Gamma}) = \mu(\Gamma) \qquad& l(\nBar{\Gamma}) = l(\Gamma) \qquad \epsilon_0(\nBar{\Gamma}) = \epsilon_1(\Gamma) \qquad \epsilon_1(\nBar{\Gamma}) = \epsilon_0(\Gamma) \label{equ:cjg1} \\
  \zeta^0(\nBar{\Gamma}) = \zeta^0(\Gamma) \qquad& \eta_0(\nBar{\Gamma}) = \eta_0(\Gamma) \qquad \zeta^1(\Gamma)+\zeta^1(\nBar{\Gamma}) = (l(\Gamma)-1) \zeta^0(\Gamma). \label{equ:cjg2}
\end{align}
\end{lem}

\begin{proof}
  Equalities in \eqref{equ:cjg1} follows directly from the definition, and assuming $\l(\Gamma)=\ell$, equalities in \eqref{equ:cjg2} follows from the observation that,
  \[ \Phi(\Gamma, \bbeta) = a_0 + \sum_{j=1}^{\ell} a_j \beta_j \quad \Longleftrightarrow \quad \Phi(\nBar{\Gamma}, \bbeta) = a_0 + \sum_{j=1}^{\ell} a_{\ell+1-j} \beta_j. \]
\end{proof}

\begin{cor}\label{cor:coeff}
For any $\delta \ge 1,$ we have 
\begin{equation}\label{equ:coeffd} \left[ d \right] Q_\delta(d) = \sum_\Gamma \mu(\Gamma) \left( -\frac{1}{2}( l(\Gamma)- \epsilon_0(\Gamma) - \epsilon_1(\Gamma)) \zeta^0(\Gamma) + \eta_0(\Gamma)\right).
\end{equation}
\end{cor}

\begin{proof}
  It is clear that conjugation defines an automorphism on the set of templates of a fixed cogenus. Hence, we can rewrite the formula \eqref{equ:coeffd0} for the coefficient of $d$ in $Q_\delta(d)$ provided in Proposition \ref{prop:coeff}:  \begin{align*}
     & \sum_\Gamma \mu(\Gamma) \left( \frac{1}{2}(2 \epsilon_1(\Gamma) - 2 l(\Gamma) +1) \zeta^0(\Gamma) + \zeta^1(\Gamma) + \eta_0(\Gamma) \right) \\
     =& \frac{1}{2} \sum_\Gamma \mu(\Gamma)  
     \begin{pmatrix}\frac{1}{2}(2 \epsilon_1(\Gamma) - 2 l(\Gamma) +1) \zeta^0(\Gamma) + \zeta^1(\Gamma) + \eta_0(\Gamma)  \\
       + \frac{1}{2}(2 \epsilon_1(\nBar{\Gamma}) - 2 l(\nBar{\Gamma}) +1) \zeta^0(\nBar{\Gamma}) + \zeta^1(\nBar{\Gamma}) + \eta_0(\nBar{\Gamma})
     \end{pmatrix}.
  \end{align*}
  Then the conclusion follows from applying Lemma \ref{lem:conjugate} to the above formula.
\end{proof}

\begin{rem}\label{rem:coeffd}
  The formula for the coefficient of $d$ in $Q_\delta(d)$ is simplified further  in \cite{toricsurface}. In \cite[Lemma 7.1]{toricsurface}, the authors show that
  \[ -\sum_\Gamma \mu(\Gamma) (l(\Gamma)- \epsilon_0(\Gamma) - \epsilon_1(\Gamma)) \zeta^0(\Gamma) = \sum_\Gamma \mu(\Gamma) \eta_0(\Gamma).\]
  As a result, the formula \eqref{equ:coeffd} for the coefficient of $d$ in $Q_\delta(d)$ can be simplified further to
  \[ \left[ d \right] Q_\delta(d) = \frac{3}{2} \sum_\Gamma \mu(\Gamma) \eta_0(\Gamma).\]
\end{rem}

Comparing Proposition \ref{prop:coeff} and Corollary \ref{cor:coeff} with \eqref{equ:genQdeltad}, we find combinatorial formulas for $A_1(t)$ and $A_2(t)$ using templates.
\begin{cor}\label{cor:A1A2}
  The power series $A_1(t)$ and $A_2(t)$ of Proposition \ref{prop:logT} are given by
\begin{align*}
  A_1(t) =& \frac{1}{2} \sum_{\delta \ge 1} \left(\sum_{\Gamma: \delta(\Gamma) = \delta } \mu(\Gamma) \zeta^0(\Gamma) \right) t^\delta \\
  A_2(t) =& \frac{1}{3} \sum_{\delta \ge 1} \left(\sum_{\Gamma: \delta(\Gamma) = \delta } \mu(\Gamma) \left( \frac{1}{2}( l(\Gamma)- \epsilon_0(\Gamma) - \epsilon_1(\Gamma)) \zeta^0(\Gamma) - \eta_0(\Gamma)\right) \right) t^\delta
\end{align*}
\end{cor}

\begin{rem}\label{rem:A2}
  If we apply the result mentioned in Remark \ref{rem:coeffd}, we obtain a simpler formula for the power series $A_2(t)$:
  \[ A_2(t) =- \frac{1}{2} \sum_{\delta \ge 1} \left(\sum_{\Gamma: \delta(\Gamma) = \delta} \mu(\Gamma) \eta_0(\Gamma) \right) t^\delta,\]
  where only involves the constant term $\eta_0(\Gamma)$ in $\Phi(\Gamma, \bbeta).$
\end{rem}

\begin{ex}

	Applying Corollary \ref{cor:A1A2} and using the data for $\zeta^0(\Gamma)$ in Table \ref{tab:templates}, we obtain the two lowest degree terms in $A_1(t)$:
\begin{align*}
  \left[ t \right]A_1(t) = & \frac{1}{2}\left(  4 \cdot 1 + 1 \cdot 2  \right) = 3, \\
  \left[ t^2 \right]A_1(t) = &  \frac{1}{2}\left( 9 \cdot 1 + 16 \cdot \left( -\frac{3}{2} \right) + 1 \cdot (-3) + 4 \cdot (-3) + 4 \cdot (-3) + 1 \cdot 3 + 1 \cdot (-3) \right) = -21.
\end{align*}
Thus, 
\[ A_1(t) = 3t - 21 t^2 + \cdots \]

Similarly, applying Corollary \ref{cor:A1A2} and/or Remark \ref{rem:A2}, we can obtain the two lowest degree terms in $A_2(t)$. (Detailed calculation is omitted.)
\[ \left[ t \right]A_2(t) = 2, \qquad \left[ t^2 \right]A_2(t) = -\frac{39}{2}.\]
Hence,
\[ A_2(t) = 2t - \frac{39}{2} t^2 + \cdots \]
\end{ex}


\section{Reformulation of the results}\label{sec:reformulate}
In Section \ref{sec:computeQ}, we proved Corollary \ref{cor:quadratic} using Corollary \ref{cor:shifttemp} and Theorem \ref{thm:main0}. Note that Corollary \ref{cor:shifttemp} is a consequence of Lemma \ref{lem:vanish0}.
The goal of this section is to prove Lemma \ref{lem:vanish0} and to finish the first reduction step in proving our main result Theorem \ref{thm:main0}.

There are two parts of this section. In the first part, we introduce {\it $\tau$-graphs}, which generalize long-edge graphs. We then extend definitions and results of long-edge graphs to $\tau$-graphs, and restate Lemma \ref{lem:vanish0} and Theorem \ref{thm:main0} in the setting of $\tau$-graphs. The description of $\tau$-graphs enables us to consider generating functions of functions $\Phi_\bbeta^s$ and $\Phi_\bbeta$ in certain forms, which will be used in the second part of this section to prove Lemma \ref{lem:vanish0} and reduce Theorem \ref{thm:main0} to a result on generating functions (Theorem \ref{thm:main1}).


We start by giving more notation that will be useful for the rest of the paer. We have been using bold letters, e.g. $\bbeta$, for vectors or vector functions. We will continue this fashion; in particular, we use $\0$ and $\1$ to denote vectors of all zeros and all ones respectively. Sometimes, we won't specify the dimensions of the vectors, which the readers should be able to figure out from the context.

We define $(\N^m)^* := \N^m \setminus \0$ to be the set of all vectors of $m$ nonnegative integers except the zero vector $\0.$ 

For any $\bn =(n_1, \dots, n_m) \in \N^m,$ we define
\[ \x^{\bn} := x_1^{n_1} x_2^{n_2} \cdots x_m^{n_m}, \quad (-1)^{\bn} := (-1)^{\sum_{i=1}^m n_i}.\]
Hence, we can write $(-\x)^{\bn}$ for $(-1)^{\bn} \x^{\bn}.$

Suppose $\bn_1 + \bn_2 + \cdots + \bn_i = \bn \in \N^m.$ We say $(\bn_1, \bn_2, \dots, \bn_i)$ is an {\it $i$-composition} of $\bn$ if $\bn_1, \dots, \bn_i \in (\N^m)^*;$ we say $(\bn_1, \bn_2, \dots, \bn_i)$ is a {\it weak $i$-composition} of $\bn$ if $\bn_1, \dots, \bn_i \in \N^m.$

\subsection{$\tau$-graphs: an alternative way of defining (long-edge) graphs}
Each edge $e$ of a (long-edge) graph $G$ contains two pieces of information: 
its weight $\rho(e)$ and its adjacent vertices. For convenience in defining the statistics $\lambda_j(G),$ we use the set $I(e) = \{a+1, a+2, \dots, b\}$ to represent the edge $\{a, b\}$ with $a <b.$ In this case, we say $e$ is of {\it type $(I(e), \rho(e))$}. We will use this representation to describe edges of our graphs. 


\begin{defn}\label{defn:taugraph}
Fixing a positive integer $m,$ let $I_1, \dots, I_m$ be subsets of $\N$ and $r_1, \dots, r_m \in \P.$ For each $1 \le i \le m,$ let $t_i = (I_i, r_i).$ We may assume $t_1, \dots, t_m$ are distinct. Let $\tau = (t_1, \dots, t_m).$ 

For any $\bn =(n_1, \dots, n_m) \in \N^m,$ we denote by $G_\tau(\bn)$ the graph on vertex set $\N$ that has $n_i$ edges of type $t_i$ for each $i.$ We call such a graph a {\it $\tau$-graph}.

Given a $\tau$-graph $G=G_\tau(\bn),$ we define its {\it multiplicity} to be
\[ \mu_\tau(\bn) := \prod_{i=1}^m \left((r_i)^2\right)^{n_i},\]
and its {\it cogenus} to be
 \[ \delta_\tau(\bn) := \sum_{i=1}^m \left( r_i |I_i|  -1 \right).\]
\end{defn}

\begin{rem}\label{rem:longedge}
Note that if we require $\tau=(t_1, \dots, t_m)$ to satisfy that for each $1 \le i \le m$,
\begin{enumerate}
\item the set $I_i$ is a set of consecutive integers, and
\item the product $r_i |I_i|$ is greater than $1,$
\end{enumerate}
we recover the definition of long-edge graphs. In particular, the definitions of multiplicity and cogenus agree with what we have defined before for long-edge graphs. Hence, $\tau$-graphs generalize long-edge graphs.

Strictly speaking, without condition (1), a $\tau$-graph is not a graph in the usual sense.
However, most of the arguments appearing in this paper work without the restrictions (1) and/or (2). 

In this paper, whenever we talk about long-edge graphs or templates, we will assume $\tau$ satisfies these two conditions without explicitly stating it.
\end{rem}

\begin{ex}\begin{enumerate}
	\item	Suppose $m=1$ and $\tau = (t_1) = ( (I, r)),$ where $I =\{1\}$  and $r \in \P.$ Then $G_\tau(n)$ is the graph with $n$ edges of weight $r$ connecting vertices $0$ and $1$.

	\item Suppose $m=2$ and $\tau = (t_1, t_2) = ( (I_1, r_1), (I_2, r_2)),$ where $I_1 = \{1\}$, $I_2 = \{1,2\}$ and $r_1, r_2 \in P.$ Then $G_\tau(n_1, n_2)$ is the graph with $n_1$ edges of weight $r_1$ connecting $0$ and $1$ and $n_2$ edges of weight $r_2$ connecting vertices $0$ and $2.$
\end{enumerate}
\end{ex}

We can naturally extend all the definitions for long-edge graphs, such as 
allowability 
and shifted graphs, to $\tau$-graphs. For convenience, we write
\begin{align*} \lambda_j(\tau,\bn) :=& \lambda_j(G_\tau(\bn)) = \sum_{i: \ j \in I_i} r_i n_i, \text{ and } \\
 \olambda_j(\tau,\bn) :=& \olambda_j(G_\tau(\bn)) = \sum_{i: \ j \in I_i} r_i n_i - \sum_{i: \ I_i = \{j\}} n_i.
 \end{align*}

\begin{defn}Suppose $\tau = (t_1, \dots, t_m),$ where $t_i = (I_i, r_i).$ Let $\maxv(\tau)$ be the largest integer appearring in $I_1, \dots, I_m,$
	  equivalently,
\[ \maxv(\tau) := \maxv(G_\tau(\bn)), \forall \bn \in (\N^*)^m.\]
\end{defn}

We also extend the concepts of (shifted) template in the following way.

\begin{defn}\label{defn:tautemplate}
  Suppose $\tau = (t_1, \dots, t_m),$ where $t_i = (I_i, r_i).$ Let $\bn \in \N^m.$ Let $\supp(\bn) := \{ i \ | \ n_i \neq 0\}$ be the set of indicies $i$ where $G_\tau(\bn)$ has edges of type $t_i.$ 
  
  We say $G_\tau(\bn)$ is a {\it $\tau$-template} if one cannot decompose $\supp(\bn)$ into two sets $S_1$ and $S_2$ such that $\cup_{i \in S_1} I_i$ and $\cup_{i \in S_2} I_{i}$ are disjoint. Otherwise, we way $G_\tau(\bn)$ is {\it not a $\tau$-template}. 
\end{defn}
One checks that if $\tau$ satisfying the conditions in Remark \ref{rem:longedge}, then $G_\tau(\bn)$ is a shifted template if and only if it is a $\tau$-template. On the other hand, any long-edge graph that is not a shifted template can be described as a $\tau$-graph (for some $\tau$) that is not a $\tau$-template.


		We have the following lemma on the function $P_\bbeta.$
\begin{lem}\label{lem:polyP}
Suppose $\tau = (t_1, \dots, t_m)$ satisfying $\maxv(\tau) \le \ell$ and $\bn \in \N^m$. Then for any $\bbeta = (\beta_1, \dots, \beta_\ell)$ satisfying $\beta_j \ge \olambda_j(\tau, \bn)$ for all $j$, the values $P_\bbeta\left(G_\tau(\bn)\right)$ are given by a multivariate polynomial in $\bbeta$ whose total degree is $|\bn| = n_1 + n_2 + \cdots + n_m,$ which is the number of edges in $G_\tau(\bn).$
\end{lem}
We will include a proof of Lemma \ref{lem:polyP} in the next section. 

Finally, we rewrite the definitions of $\Phi_\bbeta$ and $\Phi_\bbeta^s$ given in \eqref{equ:qG} and \eqref{equ:qdG}, and then restate Lemma \ref{lem:vanish0} and Theorem \ref{thm:main0} in stronger versions. Note that the definition of $P_\bbeta^s(G_\tau(\bn))$ only works if $G_\tau(\bn)$ is a long-edge graph but the defintion of $P_\bbeta(G_\tau(\bn))$ can be extended to any $\tau$-graphs.
	\begin{defn}
	Let $\bn \in (\N^*)^m.$ Define
	\begin{align}
		\Phi_\bbeta(G_\tau(\bn)) :=& \sum_{i \ge 1} \frac{(-1)^{i+1}}{i} \sum_{(\bn_1, \bn_2, \dots, \bn_i) } \prod_{j=1}^i P_\bbeta(G_\tau(\bn_j)). \label{equ:Q-P} \\
		\intertext{ and if $\tau$ satisfies the conditions in Remark \ref{rem:longedge}, also define}
		\Phi_\bbeta^s(G_\tau(\bn)) :=& \sum_{i \ge 1} \frac{(-1)^{i+1}}{i} \sum_{(\bn_1, \bn_2, \dots, \bn_i) } \prod_{j=1}^i P_\bbeta^s(G_\tau(\bn_j)), \label{equ:Q_d-P_d} 
	\end{align}
Here for both equations, the second summation is over all the $i$-compositions of $\bn$.
	\end{defn}

\begin{lem}\label{lem:vanish}
	Suppose $\tau = (t_1, \dots, t_m)$ where $t_i=(I_i,r_i)$ 
	and $\bn \in \N^m$ satisfying $G_\tau(\bn)$ is not a $\tau$-template. Then $\Phi_\bbeta(G_\tau(\bn)) =0.$

	Furthermore, if $\tau$ satisfies the conditions in Remark \ref{rem:longedge}, we have $\Phi_\bbeta^s(G_\tau(\bn)) =0.$
      \end{lem}
It follows from the comments after Defintion \ref{defn:tautemplate} that this lemma implies Lemma \ref{lem:vanish0}.

We then restate Theorem \ref{thm:main0} using the language of $\tau$-graphs.
\begin{thm}\label{thm:main}
Suppose $\tau = (t_1, \dots, t_m)$ satisfying $\maxv(\tau) \le \ell$ and $\bn \in \N^m$. Then for any $\bbeta = (\beta_1, \dots, \beta_\ell)$ satisfying $\beta_j \ge \olambda_j(\tau, \bn)$ for all $j$, the values $\Phi_\bbeta\left(G_\tau(\bn)\right)$ are given by a linear multivariate function in $\bbeta.$ 
\end{thm}


\subsection{An approach of generating functions}
We first state the following basic fact on generating functions:
suppose $f(\bn), g(\bn)$ are defined for $\bn \in (\N^m)^*.$ Then
\begin{align*}
  & g(\bn) = \sum_{i \ge 1} \frac{(-1)^{i+1}}{i} \sum_{(\bn_1, \bn_2, \dots, \bn_i): \ \text{$i$-composition of $\bn$} } \ \prod_{j=1}^i f(\bn_j), \qquad \forall \bn \\
	\Longleftrightarrow & \sum_{\bn \in (\N^m)^*} g(\bn) \x^{\bn} = \log \left( 1 + \sum_{\bn \in (\N^m)^*} f(\bn) \x^{\bn} \right)
\end{align*}

\begin{proof}[Proof of Lemma \ref{lem:vanish}]
  Without loss of generality, we may assume $\supp(\bn) = \{1, \dots, m\}$. Hence, $\bn \in \P^m.$ Further, we may assume there exists $m': 1 \le m' < m$ such that $\cup_{1 \le i \le m'} I_i$ and $\cup_{m'+1 \le i \le m} I_i$ are disjoint, and $\bn \in (\N^{m'})^* \times (\N^{m-m'})^*.$ We will show that for any $\bn' =  (\bn_1, \bn_2) \in (\N^{m'})^* \times (\N^{m-m'})^*,$ we have $\Phi_\bbeta(G_\tau(\bn')) = 0$. (So in particular, $\Phi_\bbeta(G_\tau(\bn)) = 0$.)

Recall that functions $\Phi_\bbeta$ and $P_\bbeta$ satisfy \eqref{equ:Q-P}.  Hence,
\[  \sum_{\bn' \in (\N^m)^*} \Phi_\bbeta(G_\tau(\bn')) \x^{\bn'}= \log\left(1 + \sum_{\bn' \in (\N^m)^*} P_\bbeta(G_\tau(\bn')) \x^{\bn'}\right)= \log\left(\sum_{\bn' \in \N^m} P_\bbeta(G_\tau(\bn')) \x^{\bn'}\right), \]
where the second equality follows from the fact that $G_\tau(\0)$ is the empty graph and Remark \ref{rem:emptygraph}. 

However, by the assumption of $\tau$ and the definition of $P_\bbeta$, we see that there exist functions $f_1(\bn_1)$ and $f_2(\bn_2)$ such that for any $\bn' = (\bn_1, \bn_2) \in (\N^{m'}) \times (\N^{m-m'}),$ we have
\[ P_\bbeta(G_\tau(\bn')) = f_1(\bn_1) \cdot f_2(\bn_2).\]
Let $\x = (y_1, \dots, y_{m'}, z_1, \dots, z_{m-m'}).$ Then
\begin{align*}
	\sum_{(\bn_1, \bn_2) \in (\N^m)^*} \Phi_\bbeta(G_\tau(\bn')) \y^{\bn_1} \z^{\bn_2} 
	=&  \log\left(\sum_{(\bn_1,\bn_2) \in \N^m} f_1(\bn_1) \cdot f_2 (\bn_2) \ \y^{\bn_1} \z^{\bn_2}\right) \\
	=& \log\left( \sum_{\bn_1} f_1(\bn_1) \y^{\bn_1} \right) + \log\left( \sum_{\bn_2} f_2(\bn_2) \z^{\bn_2} \right).
\end{align*}
Since the last expression only involves terms $c_{\bn_1, \bn_2} \y^{\bn_1} \z^{\bn_2}$ with one of $\bn_1$ and $\bn_2$ being zero, we've shown that $\Phi_\bbeta(G_\tau(\bn')) = 0$ for any $\bn' =  (\bn_1, \bn_2) \in (\N^{m'})^* \times (\N^{m-m'})^*.$ 

The proof of $\Phi_\bbeta^s(G_\tau(\bn))=0$ follows from exactly the same argument. Or alternatively, it also follows from Lemma \ref{lem:PPhi}, which implies that for any long edge graph $G,$ if $\Phi_\bbeta(G)=0$, then $\Phi_\bbeta^s(G)=0.$
\end{proof}

Before discussing Theorem \ref{thm:main}, we define two relevant polynomial functions. 
\begin{defn}\label{defn:twopolys}
Suppose $\tau = (t_1, \dots, t_m)$ satisfying $\maxv(\tau) \le \ell$ and $\bn \in \N^m$. Let $p_\tau(\bn, \bbeta)$ be the multivariate polynomial in $\bbeta$ 
	described in Lemma \ref{lem:polyP}. Since it is a polynomial, we can extend it to any $\bbeta \in \Z^\ell.$ 

We then define another polynomial in $\bbeta:$
\begin{equation}\label{equ:q-p} \varphi_\tau(\bn, \bbeta)  :=\sum_{i \ge 1} \frac{(-1)^{i+1}}{i} \sum_{(\bn_1, \bn_2, \dots, \bn_i) } \prod_{k=1}^i p_\tau(\bn_k, \bbeta),
\end{equation}
where the second summation is over all the $i$-compositions of $\bn$.

\end{defn}


\begin{cor}\label{cor:polyQ}
	Suppose $\tau = (t_1, \dots, t_m)$ satisfying $\maxv(\tau) \le \ell$ and $\bn \in \N^m$. Then for any $\bbeta = (\beta_1, \dots, \beta_\ell)$ satisfying $\beta_j \ge \olambda_j(\tau, \bn)$ for all $j$, the values $\Phi_\bbeta\left(G_\tau(\bn)\right)$ are given by the multivariate polynomial $\varphi_\tau(\bn,\bbeta)$.

	Furthermore, if $G_\tau(\bn)$ is not a $\tau$-template, then $\varphi_\tau(\bn,\bbeta) =0.$
\end{cor}

\begin{proof}
  Note that if $\beta_j \ge \olambda_j(\tau, \bn)$ for all $j$, for any $i$-composition $(\bn_1, \dots, \bn_i)$ of $\bn,$ we have that $\beta_j \ge \olambda_j(\tau, \bn_k)$ any $1 \le j \le \ell$ and $1 \le k \le i.$ Hence, by Lemma \ref{lem:polyP} and the defintion of $p_\tau(\bn, \bbeta),$ 
  Equation \eqref{equ:Q-P} becomes
\begin{equation} \Phi_\bbeta(G_\tau(\bn))  =\sum_{i \ge 1} \frac{(-1)^{i+1}}{i} \sum_{(\bn_1, \bn_2, \dots, \bn_i) } \prod_{k=1}^i p_\tau(\bn_k, \bbeta).
\end{equation}
Thus, the first conclusion follows.

Then the second part of the corollary follows from Lemma \ref{lem:vanish}.
\end{proof}
One sees that Theorem \ref{thm:main} just says that this multivariate polynomial $\varphi_\tau(\bn, \bbeta)$ actually is linear in $\bbeta$ for any fixed $\bn.$

Because of \eqref{equ:q-p}, it is natural to consider the following generating function
\begin{equation}\label{equ:defnsP} \sP_{\tau, \bbeta}(\x) := 1 + \sum_{\bn \in (\N^m)^*} p_\tau(\bn, \bbeta) \x^\bn.
\end{equation}
Then
\begin{equation}\label{equ:log}
\sum_{\bn \in (\N^m)^*} \varphi_\tau(\bn, \bbeta) \x^\bn = \log \left( \sP_{\tau, \bbeta}(\x) \right).
\end{equation}
We have the following theorem for the generating function $\sP_{\tau,\bbeta}(\x).$

\begin{thm}\label{thm:main1}
  Let $\tau = (t_1, \dots, t_m),$ where $t_i = (I_i, m_i)$, and $\ell$ a positive integer satisfying $\ell \ge \maxv(\tau).$ Then there exists formal power series $F_{\tau}^{(1)}(\x), F_{\tau}^{(2)}(\x), \dots, F_{\tau}^{(\ell)}(\x)$ and $H_{\tau}(\x)$ with $F_{\tau}^{(j)}(\0) =1$ for each $j$ and $H_\tau(\0)=1$ such that for any $\bbeta \in \Z^\ell,$ 
\[ \sP_{\tau, \bbeta} (\x) = \left( F_{\tau}^{(1)}(-\x) \right)^{-\beta_1-1} \cdots  \left( F_{\tau}^{(\ell)}(-\x) \right)^{-\beta_\ell-1} \cdot H_{\tau}(-\x).\]
\end{thm}

Assuming the above theorem, we can prove Theorem \ref{thm:main}.

\begin{proof}[Proof of Theorem \ref{thm:main}]
  Let $f^{(1)}(\bn), \dots, f^{(\ell)}(\bn)$ and $h(\bn)$ be the coefficients of $\x^{\bn}$ in $\log F_\tau^{(1)}(\x),$ $\dots, \log F_\tau^{(\ell)}(\x)$ and $\log H_{\tau}(\x)$ respectively, where $F_\tau^{(1)}(\x),$ $\dots, F_\tau^{(\ell)}(\x)$ and $H_{\tau}(\x)$ are the power series assumed by Theorem \ref{thm:main1}. In other words,
  \[		\log F_\tau^{(j)}(\x) = \sum_{\bn \in (\N^m)^*} f^{(j)}(\bn) \x^\bn \quad \forall j, \qquad \text{and} \quad \log H_{\tau}(\x) = \sum_{\bn \in (\N^m)^*} h(\bn) \x^\bn.\]
  Then \begin{align*}
    \log \left( \sP_{\tau, \bbeta}(\x) \right) =&  \log H_{\tau}(-\x)+ \sum_{j=1}^\ell (-\beta_j -1) \log F_\tau^{(j)}(-\x) \\
  =& \left(h(\bn)+ \sum_{j=1}^\ell (-\beta_j-1) f^{(j)}(\bn) \right) (-\x)^{\bn}.
\end{align*}
	Comparing with \eqref{equ:log}, we conclude that
	\begin{equation}\label{equ:qtau} \varphi_\tau(\bn, \bbeta) =   (-1)^{\bn}\left( - \sum_{j=1}^\ell f^{(j)}(\bn) \beta_j  + \left(h(\bn)- \sum_{j=1}^\ell f^{(j)}(\bn)\right)\right),
	\end{equation}
	which is a linear function in $\bbeta$ for any fixed $\bn.$ Hence, the conclusion follows from Corollary \ref{cor:polyQ}.
\end{proof}

Therefore, assuming Lemma \ref{lem:polyP}, we reduce the problem of proving Theorems \ref{thm:main0} and \ref{thm:main} to proving Theorem \ref{thm:main1}.

\section{Polynomiality and Reciprocity}\label{sec:polyreci}

In this section, we will prove Lemma \ref{lem:polyP} by giving an explicit formula for $p_\tau(\bn,\bbeta).$ We then introduce a new combinatorial object: {\it $(\tau, \bn)$-words}. A special family of these words, denoted by $S_\tau(\bn,\t),$ is counted by a polynomial function that has a reciprocity connection to the polynomial $p_\tau(\bn,\bbeta).$ Using this connection, we reduce our problem (of proving Theorem \ref{thm:main1}) to proving a result on the generating function of $S_\tau(\bn, \t)$ (Theorem \ref{thm:tauwords}). 

Throughout the rest of the paper, we fix $\tau = (t_1, \dots, t_m),$ where $t_i = (I_i, r_i)$, and Fix an integer $\ell \ge \maxv(\tau).$

For any $\bn = (n_1, \dots, n_m) \in \N^m,$ recall that 
\[ \lambda_j(\tau,\bn) =\sum_{i: \ j \in I_i} r_i n_i \quad \text{ and } \quad 
 \olambda_j(\tau,\bn) = \sum_{i: \ j \in I_i} r_i n_i - \sum_{i: \ I_i = \{j\}} n_i.\]
	We often omit the arguments $\tau$ and $\bn$ and only write $\lambda_j$  and $\olambda_j$ if there's no confusion.

	We start by proving a weak version of Lemma \ref{lem:polyP}.
\begin{lem}\label{lem:polyPweak}
Suppose $\tau = (t_1, \dots, t_m)$ satisfying $\maxv(\tau) \le \ell$ and $\bn \in \N^m$. Then for any $\bbeta = (\beta_1, \dots, \beta_\ell)$ satisfying $\beta_j \ge \lambda_j(\tau, \bn)$ for all $j$, (so $G_\tau(\bn)$ is $\bbeta$-allowable), the values $P_\bbeta\left(G_\tau(\bn)\right)$ are given by a multivariate polynomial in $\bbeta$ whose total degree is $|\bn| = n_1 + n_2 + \cdots + n_m,$ which is the number of edges in $G_\tau(\bn).$
\end{lem}

Before proving Lemma \ref{lem:polyPweak}, we use an example to demonstrate the basic idea of the proof.
	\begin{ex}
		Suppose $m=2$ and $\tau = (t_1, t_2) = ( (I_1, r_1), (I_2, r_2)),$ where $I_1 = \{1\}$, $I_2 = \{1,2\}$ and $r_1, r_2 \in \P.$ Let $\bn=(n_1, n_2) \in \N^2.$ Then 
		\[\lambda_1 = r_1 n_1 + r_2 n_2, \lambda_2 = r_2 n_2, \text{ and } \lambda_j = 0 \text{ for $j \ge 3.$}\]
		
		Thus, for $\bbeta =(\beta_1,\beta_2) \ge (\lambda_1, \lambda_2),$ we have that $G_\tau(\bn)$ is $\bbeta$-allowable. The graph $\ext_\bbeta(G_\tau(\bn))$ has $\beta_1-\lambda_1$ new edges connecting vertices $0$ and $1$ and $\beta_2-\lambda_2$ new edges connecting vertices $1$ and $2.$ Except the $n_2$ edges of type $t_2,$ all the other edges in $\ext_\bbeta\left( G_\tau(\bn) \right)$ have length $1$ and thus their placement between vertices is determined. Hence, we can count the total number of $\bbeta$-extended orderings (up to equivalence) by considering how many edges of type $t_2$ are placed between vertices $0$ and $1$ and how many are placed between vertices $1$ and $2.$ Therefore, we get the formula
		\[ 
		P_\bbeta\left( G_\tau(\bn) \right) = \sum_{a_{2,1} + a_{2,2} = n_2} \binom{(\beta_1-\lambda_1) + n_1 + a_{2,1}}{\beta_1-\lambda_1, n_1, a_{2,1}} \binom{(\beta_2-\lambda_2) + a_{2,2}}{a_{2,2}}. \]
In the above formula $a_{2,j}$ represents the number of edges of type $t_2$ placed between vertices $j-1$ and $j$ for $j=1,2.$ 
\end{ex}

We introduce a terminology for the data $(a_{i,j})$ used in the above example.
\begin{defn}

Let $\bn = (n_1, \dots, n_m) \in \N^m$ and $\bc =(c_1, \dots, c_\ell) \in \N^\ell$ satisfying $\sum_{i=1}^m n_i = \sum_{j=1}^\ell c_j.$ We say an $m \times \ell$ matrix $A = (a_{i,j})$ is a {\it contingency table with margin $(\bn, \bc)$} if all the entries of $A$ are nonnegative, the $i$th row sum of $A$ is $n_i$ and the $j$th column sum of $A$ is $c_j,$ i.e., the following conditions are satisfied:
\[  a_{i,j} \in \N, \quad \forall i,j; \qquad \sum_{j=1}^\ell a_{i,j} = n_i, \quad \forall 1 \le i \le m; \qquad \sum_{i=1}^m a_{i,j} = c_j, \quad \forall 1 \le j \le \ell.\]

Moreover, we say $A=(a_{i,j})$ is {\it $\tau$-compatible} if $a_{i,j} = 0$ unless $j \in I_i$.
\end{defn}

\begin{proof}[Proof of Lemma \ref{lem:polyPweak}]
	Suppose $\bbeta=(\beta_1,\dots,\beta_\ell) \ge (\lambda_1(\tau, \bn),\dots, \lambda_\ell(\tau,\bn))$. Then we have that $G_\tau(\bn)$ is $\bbeta$-allowable.

	There are two kinds of edges in $\ext_\bbeta\left( G_\tau(\bn) \right):$
	\begin{ilist}
	  \itm The original weighted edges in $G_\tau(\bn):$ for each $1 \le i \le m,$ there are $n_i$ edges of type $t_i = (I_i, r_i).$ 
	\itm The new additional unweighted edges: For each $1 \le j \le \ell,$ there are $\beta_j - \lambda_j$ new unweighted edges connecting vertices $j-1$ and $j.$
	\end{ilist}
	
	Given any $\bbeta$-extended ordering $o$ of $G_\tau(\bn),$ if for any $1 \le i \le m$ and any $1 \le j \le \ell$, let $a_{i,j}$ be the number of edges of type $t_i$ appearing between vertices $j-1$ and $j$ in the ordering $o,$ and let $c_j = \sum_{i=1}^m a_{i,j}$ be the number of all the weighted edges appearing between $j-1$ and $j,$ then the matrix $A=(a_{i,j})$ is a $\tau$-compatible contingency table of margin $(\bn, \bc),$ where $\bc = (c_1, \dots, c_\ell).$ We say $A$ is the {\it contingency table corresponding to the ordering $o$.}

	Naturally, we group $\bbeta$-extended orderings by the contingency tables they correspond to.
	Thus, we can count the number of $\bbeta$-extended orderings (up to equivalence) by:
	\begin{align*}
	& P_\bbeta(G_\tau(\bn)) \\
	=& \sum_{\bc} \sum_{A} \# \text{$\bbeta$-extended orderings (up to equivalence) corresponding to the contingency table $A$},
	\end{align*}
	where the first summation is over all the vectors $\bc = (c_1, \dots, c_\ell) \in \N^{\ell}$ satisfying $\sum c_j = \sum n_i,$ and the second summation is over all the $\tau$-compatible contingency table $A$ of margin $(\bn, \bc).$

	Fixing a contingency table $A = (a_{i,j})$ with margin $(\bn, \bc),$ we try to figure out how many ways are there to construct a corresponding $\bbeta$-extended ordering. For each $j,$ the edges between the vertices $j-1$ and $j$ include:
	\begin{itemize}
	\item $a_{i,j}$ edges of type $t_i$ for each $i: 1 \le i \le m.$
	\item $\beta_j - \lambda_j$ unweighted edges.
	\end{itemize}
	Therefore, the number of ways to order the edges between the vertices $j-1$ and $j$ is
	\[ \binom{\beta_j- \lambda_j + c_j}{c_j} \binom{c_j}{a_{1,j}, a_{2,j}, \dots, a_{m,j}}.\]
	Hence, the number of $\bbeta$-extended orderings (up to equivalence) corresponding to the contingency table $A$ is given by
	\[ \prod_{j=1}^\ell \binom{\beta_j- \lambda_j + c_j}{c_j} \binom{c_j}{a_{1,j}, a_{2,j}, \dots, a_{m,j}}.\]
	Therefore,
	\begin{equation}\label{equ:polyP}
	  P_\bbeta(G_\tau(\bn)) =  \sum_{\bc} \sum_{A}\prod_{j=1}^\ell \binom{\beta_j- \lambda_j + c_j}{c_j} \binom{c_j}{a_{1,j}, a_{2,j}, \dots, a_{m,j}}, 
	\end{equation}
	where the first summation is over all the vectors $\bc = (c_1, \dots, c_\ell) \in \N^{\ell}$ satisfying $\sum c_j = \sum n_i,$ and the second summation is over all the $\tau$-compatible contingency tables $A$ of margin $(\bn, \bc).$
	
	Clearly, this is a polynomial in $\bbeta$ whose degree is $\sum_{j=1}^\ell c_j = \sum_{i=1}^m n_i.$
\end{proof}

\begin{proof}[Proof of Lemma \ref{lem:polyP}]
	Suppose $\beta_j \ge \olambda_j$ for each $j.$ It is enough to show that $P_\bbeta(G_\tau(\bn))$ is given by the polynomial defined on the right hand side of \eqref{equ:polyP}, which is clearly true if $\beta_j \ge \lambda_j$ for each $j.$ Assume there exists $j_0$ such that 
	\[ \lambda_{j_0} > \beta_{j_0} \ge \olambda_{j_0}.\]
	Then $G_{\tau}(\bn)$ is not $\bbeta$-allowable, thus $P_\bbeta(G_{\tau}(\bn)) = 0.$ Hence, it is suffices to show that for any pair of $(\bc, A)$ in \eqref{equ:polyP}, we have 
	\begin{equation}\label{equ:binom=0} \binom{\beta_{j_0}- \lambda_{j_0} + c_{j_0}}{c_{j_0}} = 0.
	\end{equation}
	Because $\displaystyle \binom{\beta_{j_0}- \lambda_{j_0} + c_{j_0}}{c_{j_0}} = \frac{1}{c_{j_0}!} (\beta_{j_0}- \lambda_{j_0} + c_{j_0})(\beta_{j_0}- \lambda_{j_0} + c_{j_0}-1) \cdots (\beta_{j_0}- \lambda_{j_0} + 1)$ and $\beta_{j_0}- \lambda_{j_0} + 1 \le 0,$ we only need to show that $\beta_{j_0}- \lambda_{j_0} + c_{j_0} \ge 0$ to conclude \eqref{equ:binom=0}.
	However, since $A=(a_{i,j})$ is $\tau$-compatible, we have $a_{i,j_0} = n_i$ for each $i: \ I_i =\{j\}.$ One sees that 
	\[ c_{j_0} = \sum_{i=1}^m a_{i,j_0} \ge \sum_{i: \ I_i = \{j\}} n_i.\]
Therefore, 
\[  \beta_{j_0}- \lambda_{j_0} + c_{j_0} \ge \beta_{j_0}- \lambda_{j_0} + \sum_{i: \ I_i = \{j\}} n_i =  \beta_{j_0}- \olambda_{j_0}  \ge 0.\]
Then \eqref{equ:binom=0} follows.
\end{proof}

\begin{rem}\label{rem:polyPempty}
  The proof above indicates that the right hand side of \eqref{equ:polyP} defines the polynomial $p_\tau(\bn, \bbeta).$ One notices that if $\bn = \0,$ the right hand side of \eqref{equ:polyP} only has one summation term which is $1$. Hence, $p_\tau(\0, \bbeta) =1.$ This agrees with the fact that $P_\bbeta$ of $G_\tau(\0)$, an empty graph, is $1$, as we stated in Remark \ref{rem:emptygraph}.
\end{rem}

We now introduce $(\tau, \bn)$-words. Recall that we have fixed $\ell \ge \maxv(\tau).$ 
\begin{defn}\label{defn:taun}
	Fix $\bn \in \N^m.$ {\it A $(\tau, \bn)$-word} is an ordered tuple of $\ell$ words $(w_1, \dots, w_\ell)$ satisfying the following conditions:
\begin{alist}
\itm Each $w_j$ is a sequence of letters chosen from $s_0, s_1, \dots, s_m$ where repetition is allowed.
\itm For each $1 \le i \le m,$ the total number of $s_i$ appearing in all the words is $n_i.$
\itm For each $1 \le i \le m,$ the letter $s_i$ can only occur in words $w_j$ if $j \in I_i.$ 
\end{alist}

Given $L_1, \dots, L_\ell \in \N,$ we denote by $S_{\tau}(\bn; L_1, \dots, L_\ell)$ the set of all the $(\tau, \bn)$-words $(w_1, \dots, w_\ell)$ where the length of $w_j$ is $L_j$.

\end{defn}
We usually choose $\ell = \maxv(\tau).$ However, it is not hard to see that there is a natural one-to-one correspondence between the $S_\tau(n; L_1, \dots, L_\ell)$ and the set $S_\tau(n: L_1, \dots, L_\ell, L_{\ell+1}, \dots, L_{\ell'})$ for any $\ell' > \ell,$ and $L_1, \dots, L_{\ell'} \in \N,$ since for any $(w_1, \dots, w_{\ell}, w_{\ell+1}, \dots, w_{\ell'}) \in S_\tau(n: L_1, \dots, L_\ell, L_{\ell+1}, \dots, L_{\ell'})$, the word $w_j$ for $\ell < j \le \ell'$ is just a sequence of letter $s_0$'s. 
Therefore, in some sense the choice of $\ell$ is not important for the general definition of $(\tau, \bn)$-words as long as $\ell \ge \maxv(\tau).$

\begin{lem}
The cardinality of $S_{\tau}(\bn; L_1, \dots, L_\ell)$ is 
	\begin{equation*}
	  \sum_{\bc} \sum_{A}\prod_{j=1}^\ell \binom{L_j}{c_j} \binom{c_j}{a_{1,j}, a_{2,j}, \dots, a_{m,j}},
	\end{equation*}
	where the first summation is over all the vectors $\bc = (c_1, \dots, c_\ell) \in \N^{\ell}$ satisfying $\sum c_j = \sum n_i,$ and the second summation is over all the $\tau$-compatible contingency tables $A$ of margin $(\bn, \bc).$
\end{lem}
\begin{proof}
	The idea of the proof is very similar to that of Lemma \ref{lem:polyPweak}, thus is omitted.
\end{proof}

We now consider a special family of $(\tau, \bn)$-words.
\begin{defn}\label{defn:S_tau}
Fixing $\bn \in \N^m,$ for any $\t \in \N^\ell,$ we denote by $S_{\tau}(\bn, \t)$ the set of all the $(\tau, \bn)$-words $(w_1, \dots, w_\ell)$ where the length of $w_j$ is $t_j  + \lambda_j,$ i.e., 
\[ S_{\tau}(\bn, \t) := S_{\tau}(\bn; t_1+\lambda_1, t_2+\lambda_2, \dots, t_\ell + \lambda_\ell).\]
\end{defn}

\begin{cor}
The cardinality of $S_{\tau}(\bn,\t)$ is 
\begin{equation}\label{equ:polyS}
	  \sum_{\bc} \sum_{A}\prod_{j=1}^\ell \binom{t_j+\lambda_j}{c_j} \binom{c_j}{a_{1,j}, a_{2,j}, \dots, a_{m,j}},
	\end{equation}
	where the first summation is over all the vectors $\bc = (c_1, \dots, c_\ell) \in \N^{\ell}$ satisfying $\sum c_j = \sum n_i,$ and the second summation is over all the $\tau$-compatible contingency table $A$ of margin $(\bn, \bc).$

Hence, $|S_{\tau}(\bn,\t)|$ is a multivariate polynomial in $\t$ for any fixed $\bn.$
\end{cor}
Note that similar as we discussed in Remark \ref{rem:polyPempty}, we also have that $|S_\tau(\0, \t)| = 1$ for any $\t.$

\begin{defn}
	Fixing $\bn \in \N^m,$ we define $s_{\tau}(\bn, \t)$ to be the multivariate polynomial in $\t$ that computes $|S_{\tau}(\bn, \t)|$ when $\t \in \N^\ell.$ Since $s_{\tau}(\bn, \t)$ is a polynomial, we can extend it to $\t \in \Z^\ell.$

	We also define the generating function of $s_{\tau}(\bn, \t):$
\[ \cS_{\tau, \t}(\x) :=\sum_{\bn \in \N^m} s_\tau(\bn, \t) \x^\bn = 1 + \sum_{\bn \in (\N^m)^*} s_\tau(\bn, \t) \x^\bn.\]
\end{defn}

We now state the reciprocity formulas for $p_\tau(\bn, \bbeta)$ and $s_\tau(\bn, \t)$ and their generating functions, recalling that $\1$ denotes an all-one vector $(1, 1, \dots, 1).$

\begin{lem}[Reciprocity] 
	For any fixed $\bn \in \N^m,$ 
	\begin{equation}\label{equ:recip}
		p_{\tau}(\bn, \bbeta) = (-1)^{\bn} s_{\tau}(\bn, -\bbeta-\1).
	\end{equation}
Hence,
\begin{equation}\label{equ:GenRecip}
	\sP_{\tau, \bbeta}(\x) = \cS_{\tau, -\bbeta-\1}(-\x).
\end{equation}
\end{lem}

\begin{proof}
  Note that $p_{\tau}(\bn, \bbeta)$ is defined by \eqref{equ:polyP} and $s_{\tau}(\bn,\t)$ is defined by \eqref{equ:polyS}. Hence, it is enough to show that for any $\bc,$ $A$ and $\bbeta,$ we have
  \[ \prod_{j=1}^\ell \binom{\beta_j- \lambda_j + c_j}{c_j} = (-1)^{\bn} \prod_{j=1}^\ell \binom{(-\beta_j-1)+\lambda_j}{c_j}.\]
  Applying the reciprocity formula (cf. Formula (1.21) in \cite{stanleyec1ed2})
 \[ \binom{-x}{n} = (-1)^n \binom{x+n-1}{n},\]
 we get
 \[ \binom{\beta_j- \lambda_j + c_j}{c_j} =  (-1)^{c_j}  \binom{-\beta_j+\lambda_j-c_j+c_j-1}{c_j}.\]
 Hence,
 \[ \prod_{j=1}^\ell \binom{\beta_j- \lambda_j + c_j}{c_j} =  \prod_{j=1}^\ell (-1)^{c_j} \binom{(-\beta_j-1)+\lambda_j}{c_j}.\]
 However, $\sum_{j=1}^\ell c_j = \sum_{i=1}^m n_i.$ Thus, Equation \eqref{equ:recip} follows. We use \eqref{equ:recip} to prove \eqref{equ:GenRecip}:

 \begin{align*}
	 \sP_{\tau, \bbeta}(\x) =& \sum_{\bn \in \N^m} p_\tau(\bn, \bbeta) \x^{\bn} = \sum_{\bn \in \N^m} (-1)^{\bn} s_\tau(\bn, -\bbeta-\1) \x^{\bn}	\\
	 =& \sum_{\bn \in \N^m} s_\tau(\bn, -\bbeta-\1) (-\x)^{\bn} = \cS_{\tau, -\bbeta-\1}(-\x) 
	\end{align*}
\end{proof}

By \eqref{equ:GenRecip}, one sees that Theorem \ref{thm:main1} is equivalent to the following theorem: 
\begin{thm}\label{thm:tauwords}
  There exists formal power series $F_{\tau}^{(1)}(\x), F_{\tau}^{(2)}(\x), \dots, F_{\tau}^{(\ell)}(\x)$ and $H_{\tau}(\x)$ with $F_{\tau}^{(j)}(\0) =1$ for each $j$ and $H_\tau(\0)=1$ such that for any $\t \in \Z^\ell,$ 
\[ S_{\tau, \t} (\x) = \left( F_{\tau}^{(1)}(\x) \right)^{t_1} \cdots  \left( F_{\tau}^{(\ell)}(\x) \right)^{t_\ell} \cdot H_{\tau}(\x).\]
\end{thm}

Note that the functions $F_\tau^{(j)}(\x)$'s and $H_{\tau}(\x)$ in Theorem \ref{thm:tauwords} and Theorem \ref{thm:main1} are the same.

\section{Height function and decomposition of $(\tau, \bn)$-words}\label{sec:decomp}
We have reduced the problem of proving our main theorems to proving Theorem \ref{thm:tauwords}.
In this section, we will introduce two important concepts for $(\tau, \bn)$-words: height function and irreducibility, using which we prove a result on decomposing $(\tau, \bn)$-words (Theorem \ref{thm:decomp}). An analysis of the height function also leads to an alternative definition for words in $S_\tau(\bn,\t)$ (Corollary \ref{cor:heightQ}). These results will be used in the next section to prove Theorem \ref{thm:tauwords}.


\begin{defn}
	Given a $(\tau, \bn)$-word $\w = (w_1, \dots, w_\ell),$ we associate a {\it height function} with it:
	\[ h(\w) = (h_1, \dots, h_\ell) = (h_1(\w), \dots, h_\ell(\w)),\]
	where $h_j=h_j(\w)$ is defined by
	\begin{align*}
	h_j = h_j(\w) = & \quad (-1) \cdot  \# \text{(letter $s_0$'s in $w_j$)} \\
	& + \sum_{i: j \in I_i} (r_i - 1) \cdot \# \text{(letter $s_i$'s in $w_j$)} \\
		& + \sum_{i: j \in I_i} r_i \cdot \# \text{(letter $s_i$'s not in $w_j$)}. 
	\end{align*}
\end{defn}
Another way to look at the height function is that each $s_0$ appearing in $w_j$ contributes $-1$ to the height number $h_j,$ and for each $i: j \in I_i$ any letter $s_i$ appearing in $w_j$ contributes $(r_i-1)$ to $h_j$ and any letter $s_i$ appearing in words other than $w_j$ contributes $r_i$ to $h_j.$ (Note that if $s_i$ appears in $w_j$, we must have that $j \in I_i.$)

\begin{defn}
  Let $\u = (u_1, \dots, u_\ell)$ and $\v = (v_1, \dots, v_\ell)$ be a $(\tau, \bn_1)$-word and a $(\tau, \bn_2)$-word respectively. The {\it concatenation} of $\u$ and $\v$ is defined to be
  \[ \u \circ \v = (u_1 v_1, \dots, u_\ell v_\ell),\]
  which is clearly a $(\tau, \bn_1+\bn_2)$-word.

  Suppose $\w = \u \circ \v.$ We say $\u$ is an {\it initial subword} of $\w.$
\end{defn}

We give an explicit description for the height function of words in $S_\tau(\bn; L_1, \dots, L_\ell),$ followed by one basic property of the height function and an alternative defintion of $S_\tau(\bn, \t).$

\begin{lem}\label{lem:height}
Let $\w = (w_1, \dots, w_\ell) \in S_\tau(\bn; L_1, \dots, L_\ell)$. Then
\[ h_j(\w) = \lambda_j(\tau, \bn) - L_j, \quad \forall 1 \le j \le \ell.\]
\end{lem}

\begin{proof}
  Suppose $1 \le j \le \ell.$ Let $a_{i,j}$ be the number of letters $s_i$ appearing in the word $w_j$, for $1 \le i \le m$. Thus, there are $L_j - \sum_{i=1}^m a_{i,j} = L_j - \sum_{i: j \in I_i} a_{i,j}$ of $s_0$ appearing in $w_j$ and $n_i - a_{i,j}$ of $s_i$ not appearing in $w_j.$
	Then 
	\begin{align*}
		h_j(\w) = & \quad (-1) \cdot (L_j - \sum_i a_{i,j}) +   \sum_i (r_i - 1) \cdot a_{i,j}  + \sum_i r_i \cdot \left(n_i - a_{i,j}\right) \\
		= & - L_j + \sum_i a_{i,j} + \sum_i r_i a_{i,j} - \sum_i a_{i,j} + \sum_i r_i n_i - \sum_i r_i a_{i,j} \\
		=& \lambda_j(\tau, \bn) - L_j,
	\end{align*}
	where all the summations in the above equation are over all $i$ such that $j \in I_i$. 
\end{proof}

\begin{cor}\label{cor:additivity}
Suppose $\w = \u \circ \v.$ Then
\[ h(\w) = h(\u) + h(\v).\]
\end{cor}

\begin{proof}
This follows from Lemma \ref{lem:height} and the fact that
\[ \lambda_j(\tau, \bn_1+\bn_2) = \lambda_j(\tau, \bn_1) + \lambda_j(\tau, \bn_2).\]
(It is also possible to prove the corollary directly using the definition of the height function.)
\end{proof}

The following Corollary, which gives an alternative defintion for $S_\tau(\bn,\t),$ is an immediate consequence of Lemma \ref{lem:height}.
\begin{cor}\label{cor:heightQ}
Suppose $\w = (w_1, \dots, w_\ell)$ is a $(\tau, \bn)$-word and let $\t \in \N^\ell.$ Then $\w \in S_\tau(\bn, \t)$ if and only if
$h(\w) = -\t.$
\end{cor}

The above corollary tells us that the union of $S_\tau(\bn, \t)$ over all $\t \in \N^\ell$ is the set of all the $(\tau, \bn)$-words with height in $\left( \Z_{\le 0} \right)^\ell.$ For the rest of the section, we will focus more on the height of $(\tau, \bn)$-words. Thus, we will distinguish words by their height.  

\begin{defn}
	We say a $(\tau,\bn)$-word $\w$ has {\it non-positive height} if $\bh \in \left( \Z_{\le 0} \right)^\ell.$ 

	A $(\tau, \bn)$-word $\w= (w_1, \dots, w_\ell)$ is {\it balanced} if $h(\w) = \0.$


\end{defn}

\begin{rem}\label{rem:bal}
%
By Corollary \ref{cor:heightQ}, the set $S_\tau(\bn, \0)$ consists of all the balanced $(\tau,\bn)$-words. 
\end{rem}

\begin{defn}\label{defn:irred}
  Let $\t \in \N^\ell$ and $\bh = -\t.$ (So $\bh \in \left( \Z_{\le 0} \right)^\ell.$) Suppose the word $\w$ has height $\bh.$ (This means that $\w \in S_\tau(\bn, \t)$ for some $\bn \in \N^m.$)
  We say $\w$ is {\it irreducible} if it does not have a proper initial subword that also has height $\bh.$

  We denote by $S_\tau^{\irr}(\bn, \t)$ the set of all $(\tau, \bn)$-words that have height $\bh=-\t$ and are irreducible. (By Corollary \ref{cor:heightQ}, $S_\tau^{\irr}(\bn, \t)$ is a subset of $S_\tau(\bn,\t).$)
\end{defn}

\begin{rem}
Note that the only irreducible balanced words is the empty word $(\emptyset, \dots, \emptyset).$ Hence, $|S_\tau^\irr(\bn, \0)|$ is $1$ if $\bn = \0$ and is $0$ otherwise.

\end{rem}

The main result of this section is the following theorem.
\begin{thm}\label{thm:decomp}
Suppose $\bh = (h_1, \dots, h_\ell) \in \left( \Z_{\le 0} \right)^\ell$ is a non-positive height vector and $\w = (w_1, \dots, w_\ell)$ is a $(\tau, \bn)$-word satisfying $h(\w) \le \bh.$
Then $\w$ has a unique initial subword $\u$ that has height $\bh$ and is irreducible.

\end{thm}

We will prove Theorem \ref{thm:decomp} in the rest of the section. The main idea is to describe an algorithm that finds an initial subword $\u$ that has height $\bh$ and then show $\u$ is irreducible. 

\vspace*{2mm}
\noindent \textbf{Algorithm: Find-Irreducible-Subword (FIS)} 

\noindent Input: $\w = (w_1, \dots, w_\ell)$ a $(\tau,\bn)$-word and $\bh = (h_1, \dots, h_\ell) \in \left( \Z_{\le 0} \right)^\ell$ satisfying $h(\w) \le \bh.$ i.e., $h_j(\w) \le h_j$ for each $1 \le j \le \ell.$ 
\begin{enumerate}
	\item Let $\u = (u_1,\dots, u_\ell)= (\emptyset, \dots, \emptyset)$ be the empty word and let $\v =(v_1, \dots, v_\ell) = \w.$ (Clearly $h(\u) = (0, 0, \dots, 0).$) 
	\item While $h(\u) \neq \bh$:
		\begin{enumerate}
		\item Pick an index $j: 1\le j \le \ell$ such that $h_j(\u) > h_j.$ 
		\item Suppose the first letter in $v_j$ is $\alpha.$ Remove $\alpha$ from $v_j$ and append $\alpha$ to the end of $u_j.$ 
		\item Go back to (2).
		\end{enumerate}
	\item Output $(\u, \v).$
\end{enumerate}

It is clear that if the algorithm works, at each step the pair $(\u, \v)$ always has the property that $\u \circ \v = \w,$ and when the algorithm terminates, the word $\u$ satisfies that $h(\u) = \bh.$

We first show that the algorithm works.

\begin{lem}\label{lem:stepb}
At each step of the algorithm FIS, if $h_j(\u) > h_j,$ then $u_j \neq w_j.$ 
\end{lem}

This lemma indicates that whenever $h_j(\u) > h_j,$ we must have that $u_j$ is not the whole word $w_j$ yet and thus $v_j$ is not empty and we can pick the first letter of $v_j$. Hence, the algorithm won't run into trouble at step (2)/(b). We prove Lemma \ref{lem:stepb} using the following lemma.

\begin{lem}\label{lem:whole}
Suppose $\u$ is an initial subword of $\w$. For any $j: 1\le j \le \ell,$ if $u_j = w_j,$ i.e., $u_j$ is the whole word $w_j,$ then $h_j(\u) \le h_j(\w).$
\end{lem}

\begin{proof}
By the definition of the height function, if $u_j = w_j,$ we have that
\[  h_j(\w) - h_j(\u) = \sum_{i: j \in I_i} r_i \cdot \left(\#(\text{letters $s_i$'s in $\w$})- \#(\text{letters $s_i$'s in $\u$})\right) \ge 0.\]
\end{proof}

\begin{proof}[Proof of Lemma \ref{lem:stepb}]
	At each step of the algorithm, $\u$ is always an initial subword of $\w.$ By the condition of the input, we have $h_j(\w) \le h_j < h_j(\u).$ Hence, by Lemma \ref{lem:whole}, we conclude that $u_j \neq w_j.$
\end{proof}

\begin{lem}\label{lem:stepa}
At each step of the algorithm FIS, we always have that
\[ h_j(\u) \ge h_j, \quad \forall 1 \le j \le \ell.\]
\end{lem}

Note that this lemma shows that whenever $h(\u) \neq \bh,$ there always exists $j$ such that $h_j(\u) > h_j.$ Thus, the algorithm won't run into trouble at step (2)/(a).

\begin{proof}
	We prove this by induction. Initially, $h(\u) = (0, \dots, 0) \ge \bh.$
	
	Suppose at the beginning of a loop inside (2), we have $h(\u) \ge \bh.$ Since $h(\u) \neq \bh,$ we can find a $j$ such that $h_j(\u) > h_j.$ By Lemma \ref{lem:stepb}, we will run step (b) without problem. Let $\alpha$ be the letter involved. To avoid confusion, we use $\u'$ to denote the new $\u$ we obtain in step (b). We want to show that $h(\u') \ge \bh.$ There are two situations.
	\begin{itemize}
		\item If $\alpha = s_0,$ then
	\[ h_{j'}(\u') = \begin{cases}h_{j'}(\u), & j' \neq j \\
	h_{j'}(\u) -1, & j' = j
	\end{cases}.\]
		\item If $\alpha = s_i$ for some $1 \le i \le m,$ then
	\[ h_{j'}(\u') = \begin{cases}h_{j'}(\u), & j' \not\in I_i \\
	  h_{j'}(\u) +r_i-1, & j' = j \ (\text{so $j' \in I_i$}) \\
	h_{j'}(\u) +r_i, & j' \in I_i, j' \neq j
	\end{cases}.\]
	\end{itemize}
	In both cases, one checks that $h(\u) \ge \bh$ and $h_j(\u) > h_j$ imply that $h(\u') \ge \bh.$
      \end{proof}

\begin{lem}\label{lem:terminate}
The algorithm FIS always terminates.
\end{lem}

\begin{proof}
Since the number of letters in $\u$ increases by one each time we run the loop (a)-(c) inside step (2), and $\u$ is an initial subword of $\w,$ which has finitely many letters, the algorithm has to terminate at some point.
\end{proof}

Lemmas \ref{lem:stepb}, \ref{lem:stepa} and \ref{lem:terminate} show that our algorithm FIS is a well-defined algorithm. We finally discuss the properties of the output of the algorithm.

\begin{lem}\label{lem:iqb}
  Suppose $(\u^{o}, \v^o)$ is the output of the algorithm FIS taking input $(\w, \bh)$. The followings are true.
\begin{ilist}
	\itm For any initial subword $\u'$ of $\w$ that has height $\bh$, we must have that $\u^o$ is an initial subword of $\u'$. 
	\itm 
	For each $j$ where $h_j < 0,$ $u^o_j$ ends with an $s_0.$
\end{ilist}
\end{lem}
\begin{proof}
	\begin{ilist}
\itm Assume to the contrary that $\u^o$ is not an initial subword of $\u'.$ 
For convenience, we name the list of $\u$'s created by the algorithm $\u^{(0)}, \u^{(1)}, \dots, \u^{(k)}.$ So we have $\u^{(0)}$ is the empty word and $\u^{(k)}=\u^o.$

Since $\u^{(0)}$, the empty word, is an initial subword of $\u'$ by definition, there exists $i: 0 \le i \le k-1$ such that $\u^{(i)}$ is an initial subword of $\u'$ and $\u^{(i+1)}$ is not an initial subword of $\u',$ where we have obtained $\u^{(i+1)}$ from $\u^{(i)}$ by running the while loop (2) in the algorithm once. Suppose during this loop, we take the entry $h_j(\u^{(i)}) > h_j,$ and append a letter $\alpha$ to $u^{(i)}_j.$ One sees that we must have that $u^{(i)}_j = u_j'.$ Then by Lemma \ref{lem:whole},
\[ h_j(\u^{(i)})\le h_j(\u').\]
However, $h_j(\u^{(i)}) > h_j$ and $h_j(\u') = h_j.$ This is a contradiction.

\itm
One sees that $h_j(\u)$ decreases only when the algorithm adds $s_0$ to $u_j$. Further, each time we run the while loop (2) of the algorithm where a letter is appended to $u_j,$ we have to have $h_j(\u) > h_j.$ Therefore, in order to have $h_j(\u)$ become $h_j,$ a negative number, the last letter the algorithm adds to $u_j$ must be $s_0$.


\end{ilist}
\end{proof}

\begin{proof}[Proof of Theorem \ref{thm:decomp}]
	Suppose $(\u, \v)$ is the output of the algorithm FIS taking input $(\w,\bh)$.	Since $h(\u) = \bh,$ we just need to show that $\u$ is the unique irreducible initial subword with height $\bh.$ Note that any initial subword of $\u$ is an initial subword of $\w.$ Therefore, the irreducibility of $\u$ follows from Lemma \ref{lem:iqb}/(i). Suppose $\u'$ is also an irreducible initial subword of $\w$ with height $\bh.$ Then by Lemma \ref{lem:iqb}/(i), we have $\u$ is an initial subword of $\u'.$ However, since $\u'$ is irreducible, $\u = \u'.$ Thus, the uniqueness follows.
\end{proof}

Below is a consequence of Lemma \ref{lem:iqb}/(ii).

\begin{cor}\label{cor:iqb-s0}
Suppose $\w$ has height $\bh=(h_1, \dots, h_\ell)$ where $h_j=0$ or $-1$ for each $j,$ and is irreducible. Then $\w$ can be written as
\[ \w = \u \circ (s_0^{-h_1}, \cdots, s_0^{-h_\ell}),\]
for some balanced $(\tau, \bn)$-word $\u.$ Here, $s_0^k$ stands for $k$ consecutive $s_0$'s. (So $s_0^0 = \emptyset$ and $s_0^1 = s_0$.)
\end{cor}

\begin{proof}
	Since $\w$ has a non-positive height $\bh$, we can run FIS with $(\w,\bh).$ The output must be $(\w, (\emptyset, \dots, \emptyset)).$ Then the conclusion follows from Lemma \ref{lem:iqb}/(ii) and Corollary \ref{cor:additivity}.
      \end{proof}




\section{Proof of Theorem \ref{thm:tauwords}} \label{sec:finalproof}

In this section, we will use Theorem \ref{thm:decomp} to prove Theorem \ref{thm:tauwords}.
Theorem \ref{thm:decomp} is stated in terms of words of non-positive height; we need a version of it using the language of $S_\tau(\bn, \t),$ recalling $S_\tau^{\irr}(\bn,\t)$ is the set of irreducible words in $S_\tau(\bn,\t).$ (See Definition \ref{defn:irred}.)  

\begin{lem}\label{lem:decompQ}
  Suppose $\t \in \N^\ell$ and $\w \in S_\tau(\bn, \t).$ Then there is a unique way to decompose $\w$ as
\[ \w = \u \circ \v,\]
such that $\v$ is a balanced word and the word $\u$ is irreducible and has height $-\t.$

Therefore, the decomposition induces a bijection
\[ S_\tau(\bn,\t) \to \bigsqcup_{(\bn_1, \bn_0)} S_\tau^{\irr}(\bn_1, \t) \times S_\tau(\bn_0, \0), \]
where the disjoint union is over all the weak $2$-compositions $(\bn_1,\bn_0)$ of $\bn.$\end{lem}

\begin{proof}
  It follows from Corollaries \ref{cor:additivity} and \ref{cor:heightQ} and Theorem \ref{thm:decomp}.
\end{proof}

 The following is another consequence of Theorem \ref{thm:decomp}, a result on the set $S_\tau^{\irr}(\bn,\t).$
\begin{lem}\label{lem:decompirr}
  Let $\t \in \N^\ell$ and $(\t_1, \t_2, \dots, \t_k)$ a weak composition of $\t.$ Suppose $\w \in S_\tau^{\irr}(\bn, \t).$ Then there is a unique way to decompose $\w$ as
\[ \w = \u_1 \circ \u_2 \circ \cdots \circ \u_k \]
such that for each $j,$ the word $\u_j$ is irreducible and has height $-\t_j.$

Therefore, the decomposition induces a bijection
\[ S_\tau^{\irr}(\bn,\t) \to \bigsqcup_{(\bn_1, \dots, \bn_k)} S_\tau^{\irr}(\bn_1, \t_1) \times \cdots \times S_\tau^{\irr}(\bn_k, \t_k), \]
where the disjoint union is over all the weak $k$-compositions of $\bn.$
\end{lem}

\begin{proof}
  Applying Theorem \ref{thm:decomp} $k$ times, it is clear that there is a unique way to decompose $\w$ as
\[ \w = \u_1 \circ \u_2 \circ \cdots \circ \u_k \circ \v \]
such that $\v$ is a balanced word and for each $j,$ the word $\u_j$ is irreducible and has height $-\t_j.$ It is enough to show that $\v$ is the empty word. However, by Corollaries \ref{cor:additivity} and \ref{cor:heightQ}, 
\[ h(\u_1 \circ \u_2 \circ \cdots \circ \u_k) = -\t_1 - \cdots - \t_k = -\t = h(\w).\]
Since $\w$ is irreducible, the desired result follows.
\end{proof}

Recall that for $\t \in \N^\ell,$ 
\[ \cS_{\tau, \t}(\x) :=\sum_{\bn \in \N^m} |S_\tau(\bn, \t)| \x^\bn = 1 + \sum_{\bn \in (\N^m)^*} |S_\tau(\bn, \t)| \x^\bn.\]
We also define generating functions for irreducible words:
\[ \cS_{\tau, \t}^{\irr}(\x) :=\sum_{\bn \in \N^m} |S_\tau^{\irr}(\bn, \t)| \x^\bn = 1 + \sum_{\bn \in (\N^m)^*} |S_\tau^{\irr}(\bn, \t)| \x^\bn.\]
We then have the following results on these two generating functions, following Lemmas \ref{lem:decompQ} and \ref{lem:decompirr}.
\begin{cor}\label{cor:decompgen} Suppose $\t \in \N^\ell$ and $(\t_1, \dots, \t_k)$ is a weak composition of $\t.$ Then
  \[ \cS_{\tau, \t}(\x) = \cS_{\tau,\t}^{\irr}(\x) \cdot \cS_{\tau, \0}(\x),\]
  and
  \[ \cS_{\tau, \t}^{\irr}(\x) = \prod_{j=1}^k \cS_{\tau,\t_j}^{\irr}(\x).\]

\end{cor}

We are now ready to prove a weaker version of Theorem \ref{thm:tauwords}. 
\begin{prop}\label{prop:main1}
  There exists formal power series $F_{\tau}^{(1)}(\x), F_{\tau}^{(2)}(\x), \dots, F_{\tau}^{(\ell)}(\x)$ and $H_{\tau}(\x)$ with $F_{\tau}^{(j)}(\0) =1$ for each $j$ and $H_\tau(\0)=1$ such that for any $\t \in \N^\ell,$ 
\[ S_{\tau, \t} (\x) = \left( F_{\tau}^{(1)}(\x) \right)^{t_1} \cdots  \left( F_{\tau}^{(\ell)}(\x) \right)^{t_\ell} \cdot H_{\tau}(\x).\]
\end{prop}

\begin{proof}
Let $\t = (t_1, \dots, t_\ell) \in \N^\ell.$ Then 
\[ \left( \underbrace{\be_1, \dots, \be_1}_{t_1}, \underbrace{\be_2, \dots, \be_2}_{t_2}, \cdots, \underbrace{\be_\ell, \dots, \be_\ell}_{t_\ell}  \right) \]
is a weak composition of $\t,$ where $\be_j$ denotes the $j$th elementary vector of $\R^\ell$. Let
\begin{equation}\label{equ:defngen}
  F_\tau^{(j)}(\x) := \cS_{\tau, \be_j}^{\irr}(\x) \quad \forall j \qquad \text{and} \quad H_\tau(\x) := \cS_{\tau, \0}(\x).
\end{equation}
One sees that they have the desired property by Corollary \ref{cor:decompgen}.
\end{proof}

We finally prove Theorem \ref{thm:tauwords} using Proposition \ref{prop:main1}.
\begin{proof}[Proof of Theorem \ref{thm:tauwords}]
  Suppose $F_{\tau}^{(1)}(\x),$ $F_{\tau}^{(2)}(\x),$ $\dots,$ $F_{\tau}^{(\ell)}(\x)$ and $H_{\tau}(\x)$ with $F_{\tau}^{(i)}(\0) =1$ for each $i$ and $H_\tau(\0)=1$ are formal power series described in Proposition \ref{prop:main1}. It is easy to show that there exist functions $f(\bn, \t)$ for $\bn \in (\N^m)^*$ and $\t \in \Z^\ell$ such that 
  \[ \left( F_{\tau}^{(1)} \right)^{t_1} \cdots  \left( F_{\tau}^{(\ell)} \right)^{t_\ell} \cdot H_{\tau}(\x) = 1 + \sum_{\bn \in (\N^m)^*} f(\bn, \t) \x^{\bn}\quad \text{for any } \t \in \Z^\ell,\]
and $f(\bn, \t)$ is a multivariate polynomial in $\t$ for any fixed $\bn.$ 

By Proposition \ref{prop:main1}, we know that for any fixed $\bn,$ functions $f(\bn, \t)$ and $s_\tau(\bn,\t)$ are both multivariate polynomials that agree at all $\t \in \N^\ell.$ However, $\N^\ell$ is a Zariski dense subset of $\C^\ell.$ Therefore, they have to be the same polynomial. Hence, Theorem \ref{thm:tauwords} follows.
\end{proof}

\begin{rem}
  By the proofs in this section, one sees that in order to figure out the coefficients of linear function described in Theorems \ref{thm:main0intro} and \ref{thm:main0}, it is sufficient to the find the generating functions for $|S_{\tau}^{\irr}(\bn, \be_j)|$ for each $j$ and the generating function for $|S_\tau(\bn, \0)|$.

  By Corollary \ref{cor:iqb-s0}, the set $S_{\tau}^{\irr}(\bn, \be_j)$ is in bijection with a subset of $S_\tau(\bn,\0).$ So it might be worth studying the property of this particular subset. 

  Recall that $S_\tau(\bn, \0)$ is just the set of balanced $(\tau, \bn)$-words. Hence, balanced $(\tau, \bn)$-words could be an interesting subject for future study. 
\end{rem}

\section{Examples of $\varphi(\bn,\bbeta)$}\label{sec:examples}

Recall that $\varphi_\tau(\bn, \bbeta)$ is the linear function described in Theorems \ref{thm:main0intro} and \ref{thm:main0} if $G = G_\tau(\bn).$ 
In this section, we first consider a family of simple examples for which we are able to describe $p_\tau(\bn, \bbeta)$, $|S_\tau^{\irr}(\bn, \be_1)|$ and $|S_\tau(\bn, \0)|$ explicitly, and demonstrate the idea of how one might use this information to figure out an expression for $\varphi_\tau(\bn,\bbeta)$ on a special subfamily of the examples. We then extend the result and give an expression for $\varphi_\tau(\bn,\bbeta)$ for any $\tau$ that consists of one type of edges. 

We start by considering the situation when $\maxv(\tau) = 1.$ Then every $I_i =\{1\}.$ Suppose $\tau = (t_1, t_2, \dots, t_m),$ where $t_i = (\{1\}, r_i)$ and $r_1, \dots, r_m$ are distinct positive integers. Then $G_\tau(\bn)$ has $\sum_i n_i$ edges, $n_i$ of which have weight $r_i.$ We have $\lambda_1\left( G_\tau(\bn) \right) = \sum_i r_i n_i$ and $\lambda_j\left( G_\tau(\bn) \right)  = 0$ for all $j \ge 2.$
Since $\maxv(G_\tau(\bn)) \le 1,$ we consider $\bbeta = \beta \in \Z$ a $1$-dimensional vector.

For $\beta \ge \sum_i r_i n_i,$ the set of edges connecting vertices $0$ and $1$ in the graph $\ext_\beta\left(G_\tau(\bn)\right)$ consists of $n_i$ edges of weight $r_i$ for each $i$ and $\beta - \sum_i r_i n_i$ unweighted edges. Therefore,
\[ p_\tau(\bn, \beta) = P_\beta\left(G_\tau(\bn)\right) = \binom{\beta-\sum_i r_i n_i + \sum_i n_i}{\sum_i n_i} \binom{\sum_i n_i}{n_1, n_2, \dots, n_m}.\]
Thus,
\[ \sP_{\tau, \beta}(\x) = 1 + \sum_{\bn \in (\N^m)^*} \binom{\beta-\sum_i r_i n_i + \sum_i n_i}{\sum_i n_i} \binom{\sum_i n_i}{n_1, n_2, \dots, n_m} \x^{\bn},\]
and
\[ \sum_{\bn \in (\N^m)^*}  \varphi_\tau(\bn, \beta)\x^{\bn} = \log \left( 1 + \sum_{\bn \in (\N^m)^*} \binom{\beta-\sum_i r_i n_i + \sum_i n_i}{\sum_i n_i} \binom{\sum_i n_i}{n_1, n_2, \dots, n_m} \x^{\bn} \right).\]

Let $\ell = \maxv(\tau) = 1.$
In order to find functions $F_\tau^{(1)}(\x)$ and $H_{\tau}(\x)$ in Theorem \ref{thm:main1}, Theorem \ref{thm:tauwords} and Proposition \ref{prop:main1}, we consider the corresponding $(\tau, \bn)$-words. As we discussed in the proof of Proposition \ref{prop:main1}, $F_\tau^{(1)}(\x) = \cS_{\tau, \be_1}^{\irr}(\x)$ which is the generating function for $S_\tau^{\irr}(\bn, \be_1),$ and $H_\tau(x) = \cS_{\tau, \0}(x)$ which is the generating function for $S_\tau(\bn, \0).$ Note that since we are in a $1$-dimensional space, $\be_1 = 1$ and $\0 = 0.$

For the given setup, the $(\tau, \bn)$-words are actually the Lukasiewicz words in the literature. See Section 5.3 of \cite{stanleyec2}. It is known that there is a natural one-to-one correspondence between words $\w$ in $S_\tau^{\irr}(\bn, 1)$ and plane trees which have $n_i$ internal vertices of degree $r_i$ for each $i$ and $1+\sum_i (r_i-1) n_i$ leaves. Using this, one obtains the cardinality for $S_\tau^{\irr}(\bn, 1)$ (Theorem 5.3.10 in \cite{stanleyec2}):
\begin{align*}
	|S_\tau^{\irr}(\bn, 1)| =& \frac{1}{1+\sum_i r_i n_i} \binom{1 + \sum_i r_i n_i}{\sum_i n_i} \binom{\sum_i n_i}{n_1, n_2, \dots, n_m} \\
	=& \frac{1}{1+\sum_i (r_i-1) n_i} \binom{\sum_i r_i n_i}{\sum_i n_i} \binom{\sum_i n_i}{n_1, n_2, \dots, n_m}.
\end{align*}
Next, recall that the set $S_\tau(\bn, 0)$ is just the set of all the balanced words. Hence, $w \in S_\tau(\bn, 0)$ if and only if $w$ contains $n_i$ copies of the letter $s_i$ for each $i>0$ and $\sum_i (r_i-1) n_i$ copies of the letter $s_0$. Therefore,
\[ |S_\tau(\bn, 0)| = \binom{\sum_i r_i n_i}{\sum_i n_i} \binom{\sum_i n_i}{n_1, n_2, \dots, n_m}.\]

Therefore, we let
\begin{align*}
	F_\tau^{(1)}(\x) =& \sum_{\bn \in \N^m} \frac{1}{1+\sum_i (r_i-1) n_i} \binom{\sum_i r_i n_i}{\sum_i n_i} \binom{\sum_i n_i}{n_1, n_2, \dots, n_m} \x^{\bn}, \\
	H_{\tau}(\x) =& \sum_{\bn \in \N^m}  \binom{\sum_i r_i n_i}{\sum_i n_i} \binom{\sum_i n_i}{n_1, n_2, \dots, n_m} \x^{\bn}.
\end{align*}
Then these are the functions in Theorem \ref{thm:tauwords} and Theorem \ref{thm:main1}. Hence,
\begin{align*}
	\cS_{\tau, t}(\x) =& \left( F_\tau^{(1)}(\x) \right)^t H_{\tau}(\x), \text{ and } \\
	\sP_{\tau, \beta}(\x) =& \left( F_\tau^{(1)}(-\x) \right)^{-\beta-1} H_{\tau}(-\x).
\end{align*}
Thus, if we let $f(\bn)$ and $h(\bn)$ be the coefficients of $\x^{\bn}$ in $\log F_\tau^{(1)}(\x)$ and $\log H_{\tau}(\x)$ respectively, 
	then as in the proof of Theorem \ref{thm:main}, we obtain 
	\begin{equation}\label{equ:varphitau} \varphi_\tau(\bn, \beta) =   (-1)^{\bn}\left( -f(\bn) \beta + (h(\bn)-f(\bn))\right). 
	\end{equation}

	One way to figure out $f(\bn)$ and $h(\bn)$ is to use the Lagrange inversion formula. For example, it is known that $z := F_\tau^{(1)}(\x)$, the generating function for Lukasiewicz words, satisfies the equation:
\[ z = 1 + \sum_{i=1}^m x_i z^{r_i}.\]
Hence, one can use multivariate Lagrange inversion formula to figure out $f(\bn)$, which is the coefficient of $\x^{\bn}$ of $\log(z).$ 
Since the description of multivariate Lagrange inversion is complicated, we only demonstrate this approach with a special situation: when $m=1$ and $\tau= (t_1) = ( (\{1\}, r ))$ where $r \in \P,$ in which only single variate Lagrange inversion is needed.

Suppose ${\tau = ( (\{1\}, r) )}$. Since this is a special case of what we've discussed above, we immediately have
\begin{align}
  \sP_{\tau,\beta}(x) =& 1 + \sum_{n \in \P} \binom{\beta-rn + n}{n} x^n = \left(F_\tau^{(1)}(x)\right)^{-\beta-1} H_{\tau}(x), \label{equ:excP} \\
	\text{where} \quad	F_\tau^{(1)}(x) =& \sum_{n \in \N} \frac{1}{1+(r-1) n} \binom{rn}{n} x^{n}, \nonumber \\
	H_{\tau}(x) =& \sum_{n \in \N}  \binom{rn}{n} x^{n}. \nonumber
\end{align}
(In fact, words in $S_\tau^{\irr}(n, 1)$ are in one-to-one correspondence to $r$-ary trees with $n$ internal vertices, whose cardinality is well-known to be $\ds \frac{1}{1+(r-1) n} \binom{rn}{n}.$)  

For convenience, we abbreviate $F_\tau^{(1)}$ and $H_{\tau}(x)$ to $F(x)$ and $H(x)$ respectively.  By Examples 6.2.6 and 6.2.7 in \cite{stanleyec2}, the generating functions $F$ and $H$ satisfy the following equations:
\begin{align}
	F =& 1 + F G^r \label{equ:G} \\
	H(x^{r-1}) =& \frac{d}{dx} \left( x F(x^{r-1}) \right) = F(x^{r-1}) + x F'(x^{r-1}) (r-1) x^{r-2} \label{equ:G2H1}\\
	=& F(x^{r-1}) + (r-1) x^{r-1} F'(x^{r-1}) \nonumber
\end{align}
Note that \eqref{equ:G2H1} implies that
\begin{equation}\label{equ:G2H2}
	H = F + (r-1) x F'.
\end{equation}
By differentiating \eqref{equ:G}, we get
\[ F' = F^r + x r F^{r-1} F' \quad \Rightarrow \quad F' = \frac{F^r}{1 - x r F^{r-1}}.\]
Plugging the formula for $F'$ into \eqref{equ:G2H2}, we obtain
\begin{equation}\label{equ:G2H3}
	H = F +  \frac{x (r-1) F^r}{1 - x r F^{r-1}} = \frac{F - xF^r}{1- xr F^{r-1}} = \frac{1}{1 - xr F^{r-1}} = \frac{F}{F - xr F^r} = \frac{F}{F- r(F-1)}, 
\end{equation}
where both the third and the fifth equalities follow from \eqref{equ:G}.

Let \[ z = F_\tau^{(1)}(x)-1 = \sum_{n \in \P} \frac{1}{1+(r-1) n} \binom{rn}{n} x^{n}.\]
Then \eqref{equ:G} becomes $z = x(z+1)^r,$ which is equivalent to \[z = \left(\ds \frac{x}{(1+x)^r} \right)^{<-1>}.\]
Hence, using the Lagrange inversion formula \cite[Corollary 5.4.3]{stanleyec2}, we find that $f(n)$, the coefficient of $x^n$ in $\log (1+z)$, is given by
\begin{equation}\label{equ:gr} f(n) = \frac{1}{n} [t^{n-1}] \left(\frac{1}{1+t} \left((1+t)^r\right)^n\right) = \frac{1}{n} \binom{rn-1}{n-1} = \frac{1}{rn} \binom{rn}{n}.
\end{equation}
We also rewrite \eqref{equ:G2H3} using $z$:
\[ H = \frac{F}{F- r(F-1)} = \frac{1+z}{1+z - rz} = \frac{1+z}{1- (r-1)z}.\]
Hence,
\[ \log H = \log (1+z) - \log (1 - (r-1) z).\]
Therefore, using the Lagrange inversion formula, we find that $h(n),$ the coefficient of $x^n$ in $\log (1+z) - \log (1- (r-1)z)$, is given by
\begin{align} h(n) =& \frac{1}{n} [t^{n-1}] \left( \frac{1}{1+t} + \frac{1}{1-(r-1)t}  \right) \left( (1+t)^r \right)^n \nonumber  \\
	=& f(n) + \frac{1}{n} [t^{n-1}] \left( (1+t)^{rn} \sum_{i \ge 0} (r-1)^i t^i \right)\nonumber \\
	=& f(n) + \frac{1}{n}  \sum_{i = 0}^{n-1} \binom{rn}{i} (r-1)^{n-1-i}\label{equ:hr}  
\end{align}
Therefore, applying \eqref{equ:gr} and \eqref{equ:hr} to \eqref{equ:varphitau}, we obtain the following result.

\begin{lem}
	Suppose $m=1$ and $\tau = \left( (\{1\}, r) \right).$ Then
	\[ \varphi_\tau(n, \beta) = \frac{(-1)^{n+1}}{n} \left(\frac{1}{r} \binom{rn}{n} \beta -   \sum_{i = 0}^{n-1} \binom{rn}{i} (r-1)^{n-1-i} \right).\]
\end{lem}

Recall for the setup in the above lemma, $\varphi_\tau(n, \beta)$ is the coefficient of $x^n$ in the logarithm of the generating function $\cP_{\tau, \beta}(x)$ defined in \eqref{equ:excP}.
Hence, the above lemma is equivalent to the following lemma.
\begin{lem}\label{lem:equivalgen}
For any unknown $\beta,$ and any positive integer $r,$ the coefficient of $x^n$ in 
	\[ \log \left(1 + \sum_{n \in \P} \binom{\beta-rn + n}{n} x^n\right)\]
	is given by
	\[ \frac{(-1)^{n+1}}{n} \left(\frac{1}{r} \binom{rn}{n} \beta -   \sum_{i = 0}^{n-1} \binom{rn}{i} (r-1)^{n-1-i} \right).\]
      \end{lem}

	Therefore, we are actually able to compute $\varphi_\tau(n, \bbeta)$ for any $\tau$ that contains only one type of edges.

	\begin{lem}\label{lem:m=1}
	Suppose $m=1$ and $\tau = ( ( \{1, 2, \dots, \ell\}, r)).$ Then
	\[ \varphi_\tau(n, \bbeta) = \frac{(-1)^{n+1}}{n} \left(\frac{1}{r \ell} \binom{r \ell n}{n} \left( \left(\sum_{j=1}^\ell \beta_j \right) + (\ell-1)\right)  -   \sum_{i = 0}^{n-1} \binom{r\ell n}{i} (r\ell-1)^{n-1-i} \right).\]
	\end{lem}

	\begin{proof}
	It is clear that we have
	\[ \lambda_j(G_\tau(n)) = \begin{cases} rn & \text{if $1 \le j \le \ell$}\\  0 & \text{if $j > \ell$.} \end{cases} \]
	  For any $\bbeta = (\beta_1, \dots, \beta_\ell) \ge (rn, \dots, rn),$ to create $\ext_\bbeta\left( G_\tau(n) \right),$ we add $\beta_j - rn$ unweighted edges connecting vertices $j-1$ and $j$ for any $1 \le j \le \ell.$ In total, we have 
	  \[ \sum_{j=1}^\ell \left( \beta_j - rn\right) = \left(\sum_{j=1}^\ell \beta_j \right)- r\ell n  \]
	unweighted edges between vertices $0$ and $\ell.$

	One sees that $P_\bbeta\left( G_\tau(n) \right)$ is the number of ways to place the $n$ edges of weight $r$ between the vertices $0$ and $\ell.$ Note that the other elements that are between vertices $0$ and $\ell$ are these $\left(\sum_{j=1}^\ell \beta_j \right) - r\ell n$ unweighted edges and $\ell-1$ vertices: $1, 2, \dots, \ell-1,$ where the order of these two kinds of elements are fixed in any $\bbeta$-extended ordering. Hence, we conclude that
	\[ p_\tau(n,\bbeta) = P_\bbeta\left( G_\tau(n) \right) = \binom{\left(\sum_{j=1}^\ell \beta_j \right)  - r\ell n + (\ell -1) + n}{n}.\]
	Therefore, $\varphi_\tau(n,\bbeta)$ is the coefficient of $x^n$ in the generating function
	\[ \log\left(1 + \sum_{n \in P} p_\tau(n,k) x^n\right) = \log\left(1 + \sum_{n \in \P} \binom{ \left(\sum_{j=1}^\ell \beta_j \right)+ (\ell -1)  - r\ell n + n}{n} x^n\right).\]
	Then the conclusion follows from Lemma \ref{lem:equivalgen}.
	\end{proof}

	\begin{ex}
	  Suppose $\tau = ( ( \{1, 2\}, r)).$ Then $G_\tau(n)$ is the long-edge graph with $n$ edges of weight $r$ connecting vertices $0$ and $2$. By Lemma \ref{lem:m=1}, we have
	  \begin{align*} \varphi_\tau(n,\bbeta	) =&  \frac{(-1)^{n+1}}{n} \left(\frac{1}{2r} \binom{2 r n}{n} \left( \beta_1+\beta_2 + 1 \right)  -   \sum_{i = 0}^{n-1} \binom{2 r n}{i} (2r-1)^{n-1-i} \right) 
	  \end{align*}
Assume further that $r=1.$ Then
\begin{align}
	\varphi_\tau(n,\bbeta) =&  \frac{(-1)^{n+1}}{n} \left(\frac{1}{2}\binom{2 n}{n} \left( \beta_1+\beta_2  + 1\right)  -   \sum_{i = 0}^{n-1} \binom{2 n}{i} \right) \nonumber \\
	=& \frac{(-1)^{n+1}}{n} \left(\frac{1}{2}\binom{2 n}{n} \left( \beta_1 + \beta_2 \right) + \binom{2n}{n}  -   2^{2n-1} \right), 	 \label{equ:m1l2r1} 
\end{align}
where the second equality follows from the identity 
\[ \sum_{i=0}^{n-1} \binom{2n}{i} = \frac{1}{2} \left( \sum_{i=0}^{n-1} \binom{2n}{i} + \sum_{i=n+1}^{2n} \binom{2n}{i} \right) = \frac{1}{2} \left( \sum_{i=0}^{2n} \binom{2n}{i} -  \binom{2n}{n} \right) = \frac{1}{2} \left( 2^{2n} - \binom{2n}{n} \right).\]

In particular, plugging $n=1$ and $2$ in \eqref{equ:m1l2r1} we get
\begin{align*}
	\varphi_\tau(1,\bbeta) =& \frac{(-1)^{1+1}}{1} \left(\frac{1}{2}\binom{2}{1} \left( \beta_1+\beta_2 \right) + \binom{2}{1}  -   2^{2-1} \right) = \beta_1 + \beta_2, 	  \\
	\varphi_\tau(2,\bbeta) =& \frac{(-1)^{2+1}}{2} \left(\frac{1}{2}\binom{4}{2} \left( \beta_1+\beta_2\right) + \binom{4}{2} -   2^{4-1} \right)=   -\frac{1}{2}(3\beta_1 + 3 \beta_2 - 2).
\end{align*}
Note that $G_\tau(1)$ and $G_\tau(2)$ are the second and fifth templates which appeared in Table \ref{tab:templates}. The functions we obtain agree with those listed in Table \ref{tab:templates}.
	\end{ex}

\bibliographystyle{amsplain}
\bibliography{gen}

\end{document}